\documentclass[12pt]{article}
\usepackage{latexsym, amssymb, amsthm, amsmath}
\usepackage{dsfont}
\usepackage{graphicx}
\usepackage{color}

\DeclareMathAlphabet\gothic{U}{euf}{m}{n}

\makeatletter
\def\eqnarray{\stepcounter{equation}\let\@currentlabel=\theequation
\global\@eqnswtrue
\tabskip\@centering\let\\=\@eqncr
$$\halign to \displaywidth\bgroup\hfil\global\@eqcnt\z@
  $\displaystyle\tabskip\z@{##}$&\global\@eqcnt\@ne
  \hfil$\displaystyle{{}##{}}$\hfil
  &\global\@eqcnt\tw@ $\displaystyle{##}$\hfil
  \tabskip\@centering&\llap{##}\tabskip\z@\cr}

\def\endeqnarray{\@@eqncr\egroup
      \global\advance\c@equation\m@ne$$\global\@ignoretrue}

\def\@yeqncr{\@ifnextchar [{\@xeqncr}{\@xeqncr[5pt]}}
\makeatother

\allowdisplaybreaks[3]

\newtheorem{theorem}{Theorem}[section]
\newtheorem{thm}[theorem]{Theorem}
\newtheorem{lemma}[theorem]{Lemma}
\newtheorem{lem}[theorem]{Lemma}
\newtheorem{corollary}[theorem]{Corollary}
\newtheorem{cor}[theorem]{Corollary}

\newtheorem{prop}[theorem]{Proposition}

\theoremstyle{definition}

\newtheorem{defn}[theorem]{Definition}
\newtheorem{assu}[theorem]{Assumption}

\newtheorem{exam}[theorem]{Example}

\newtheorem{rem}[theorem]{Remark}

\theoremstyle{remark}

\newcommand{\gots}{\gothic{s}}

\newcounter{teller}

\newenvironment{tabel}{\begin{list}%
{\rm  (\alph{teller})\hfill}{\usecounter{teller} \leftmargin=1.1cm
\labelwidth=1.1cm \labelsep=0cm \parsep=0cm \setlength{\listparindent}{\parindent}}
                      }{\end{list}}

\newcounter{tellerr}

\newenvironment{tabeleq}{\begin{list}%
{\rm  (\roman{tellerr})\hfill}{\usecounter{tellerr} \leftmargin=1.1cm
\labelwidth=1.1cm \labelsep=0cm \parsep=0cm \setlength{\listparindent}{\parindent}}
                         }{\end{list}}

\newcounter{tellerrr}

\newcounter{proofstep}

\newcommand{\Ni}{\mathds{N}}

\newcommand{\Ri}{\mathds{R}}
\newcommand{\Ci}{\mathds{C}}

\newcommand{\field}[1]{\mathbb{#1}}

\newcommand{\R}{\field{R}}
\newcommand{\C}{\field{C}}

\newcommand{\RRe}{\mathop{\rm Re}}

\newcommand{\supp}{\mathop{\rm supp}}

\newcommand{\mr}{{\rm MR}}

\newcommand{\dom}{\operatorname{Dom}}

\newcommand{\ca}{{\cal A}}
\newcommand{\cb}{{\cal B}}

\newcommand{\ce}{{\cal E}}

\newcommand{\ci}{{\cal I}}
\newcommand{\cj}{{\cal J}}
\newcommand{\ck}{{\cal K}}
\newcommand{\cl}{{\cal L}}

\newcommand{\fD}{\mathfrak D}
\newcommand{\fI}{\mathfrak I}

\begin{document}
\bibliographystyle{tom}

{\bf On maximal parabolic regularity for non-autonomous parabolic operators}

\begin{center}
Karoline Disser, A.F.M. ter Elst, Joachim Rehberg
\end{center}

\begin{abstract} 
	We consider linear inhomogeneous non-autonomous parabolic problems associated to sesquilinear 
forms, with discontinuous dependence of time. 
We show that for these problems, the property of maximal parabolic regularity can be 
extrapolated to time integrability exponents $r \neq 2$.
This allows us to prove maximal parabolic $L^r$-regularity for discontinuous non-autonomous 
second-order divergence form operators in very general geometric settings 
and to prove existence results for related quasilinear equations.  
\end{abstract}

\par
\emph{Acknowledgements.} K.D. was supported by the European Research 
Council via ``ERC-2010-AdG no. 267802 (Analysis of Multiscale Systems 
Driven by Functionals)''.

\tableofcontents

\vspace*{10mm}

\newpage

\section[Introduction] {Introduction} \label{Snonaut1}

In this paper we are 
interested in maximal parabolic regularity for 
 non-autonomous parabol\-ic equations like
\[
u'(t) +\ca(t) u(t) =f(t),  
\]
for almost every $t \in (0,T)$,
where $u(0)=0$, $f \in L^r((0,T);X)$ and 
the operators $\ca(t)$ all have the same domain of definition $D$ in a Banach space~$X$.
If the operator function $\ca(\cdot)$
is constant, these equations may be solved using the concept of maximal parabolic regularity, see 
for example
\cite{DeS}, \cite{ArendtBu}, \cite{Are6}, \cite{Lamb}, 
\cite{DV}, \cite{dor}, \cite{Weis}, \cite{DHP},  \cite{ABHR}, 
\cite{DisserMeyriesRehberg}, \cite{PrSi}.
This theory extends to cases in which the dependence
\begin{equation} \label{e-discont}
(0,T)=J \ni t \mapsto \ca(t) \in \cl(D;X)
\end{equation}
is continuous, see \cite[Chapter~VI]{Gris2}, \cite{PrussSchnaubelt}, 
\cite{Ama5} and the survey in \cite{Schn} and it is a powerful tool for solving 
corresponding nonlinear equations,  see  \cite{CL}, \cite{Pru2},
 \cite{Ama6}, \cite{HiebR}, \cite{HaR}, \cite{HaR3}, 
\cite{KPW}, \cite{EMR}.

If the continuity of \eqref{e-discont} is violated, things are much less understood, 
the only classical exception being 
 the case that the summabilility exponent in 
time, $r$, equals $2$ and that $X$ and $D$ are Hilbert spaces (see Proposition \ref{p-Lions} below). 
For results in the Banach space case, we refer to \cite{AltLuck} and \cite{Groe2} and for 
relevant recent achievements see \cite{ACFP}, \cite[Proposition~1.3]{ACFP}, \cite{ADLO}, 
\cite{GallaratiVeraar}, \cite{Fackler1}, \cite{Fackler2}
and the references therein. 
These results are mostly to be seen  as perturbation results, 
with respect to an \emph{autonomous} parabolic operator. 

The spirit of our paper is perturbative in a different sense: not the operator is changed, but the Banach space.
This enables us to extrapolate
maximal parabolic regularity for whole classes of non-autonomous operators simultaneously. 
Remarkably, Gr\"oger proved in \cite{Groe2} that maximal parabolic
$L^r(J;W^{-1,q}_\fD)$-regularity for second-order divergence-form operators is preserved in case 
of non-smooth, time-dependent coefficients,
if one deviates from $q=r=2$ to temporal and spatial integrability exponents $q=r$ in an interval 
$[2,2+\varepsilon)$ that 
depends on the ellipticity
constant and the $L^\infty$-norm of the coefficient function.
Here, we provide an abstract 
extrapolation strategy that  includes  
general non-autonomous operators corresponding to sesquilinear forms, 
and we extend the results of Gr\"oger to indices $q\neq r \in (2-\varepsilon, 2+ \varepsilon)$ 
and more general geometric settings for mixed boundary conditions.

In the first part of the paper, we develop an abstract framework  which allows to show 
maximal parabolic $L^r(J;X)$-regularity for \emph{non-autonomous} operators in 
a Banach space $X$ and for some $r\in (1,\infty)$, provided that one knows 
maximal parabolic regularity for suitable \emph{autonomous} operators. 
Later on we specify $X$ to be $W_\fD^{-1,q}(\Omega)$
(see Definition \ref{d-1} below) and $\ca(t)$ to be a second-order divergence
operator $-\nabla \cdot \mu_t \nabla +I$.
We aim at situations in which the map in \eqref{e-discont} is substantially discontinuous, 
which means that it
is allowed to be discontinuous in every point $t$ of the
time interval $J$, and does \emph{not} satisfy the already weak condition of relative 
continuity in \cite{ACFP}.
A prototype for this is the following: for each time
$t \in J$ there is a moving subdomain $\Omega_t \subset \Omega$ on which the coefficient function
$\mu_t$ is constant and it takes a different constant value on $\Omega \setminus \Omega_t$, 
see Section~\ref{Snonaut8}.

The Banach spaces $X$ of type $W^{-1,q}_\fD(\Omega)$ turn out 
to be well suited for the treatment of 
elliptic and parabolic problems if these are combined with inhomogeneous Neumann 
boundary conditions (cf.\ \cite[Section~3.3]{Lio3} for $q=2$) or if right 
hand sides of distributional type appear, e.g.\ surface densities, which may even
vary their position in time, cf.\ \cite{HaR}, \cite{HaR3}.
Note that there is
often an 
intrinsic connection between (spatial) jumps in the coefficient function
and the appearance of surface densities as parts of the right hand side 
(see \cite[Chapter~1]{Tam}).
Interestingly, these spaces are also adequate for the
treatment of control problems, see \cite{KunischPieperVexler}, \cite{CasasClasonKunisch}, \cite{KrumbiegelRehberg}, \cite{HMRR}.
The advantage of $W^{-1,q}_\fD(\Omega)$ over $W^{-1,2}_\fD(\Omega)$
is that the domain of elliptic divergence operators continuously embeds into
a H\"older space if $q$ is larger than the space dimension (cf.\ \cite{HMRS} \cite{ERe2}),
 which is helpful when considering  quasilinear problems, see Section~\ref{Snonaut00} below and 
cf.\ \cite{Pru2}, \cite{HiebR}, \cite{HaR} in the continuous setting. 
Moreover, in contrast to the $L^p(\Omega)$ spaces, the space $W^{-1,q}_\fD(\Omega)$
satisfies the property that the domains of divergence operators $-\nabla \cdot (\mu_t \nabla)$
coincide at different points in time even if the discontinuities of the coefficient functions $\mu_t$ move
in $\Omega$, see the examples in Section~\ref{Snonaut6}, taken from
\cite{ElschnerRS}, \cite{DisserKaiserRehberg}.

We next give an outline of the paper.
We first recall preliminary results on maximal parabolic regularity and a 
quantitative version of Sneiberg's extrapolation principle. 
Then we prove an interpolation result for different spaces of maximal parabolic
regularity, i.e.\ we prove the interpolation identity
\begin{equation} \label{e-interpolident0}
[L^{r_0}(J;D) \cap W_0^{1,r_1}(J;X),L^{r_1}(J;E) \cap W_0^{1,r_2}(J;Y)]_\theta 
=
L^r(J;[D,E]_\theta) \cap W_0^{1,r}(J;[X,Y]_\theta),
\end{equation}
in which $\frac{1}{r}=\frac{1-\theta}{r_0} +\frac{\theta}{r_1}$ and 
$(D,X)$ and $(E,Y)$ form a \emph{pair of common
maximal parabolic regularity}, see Definition \ref{d-commonpair} below.
Having this at hand, one can extrapolate maximal parabolic regularity from one
setting to `neighbouring ones', see Theorem~\ref{t-extrapolatabstr}.

Then in Section \ref{Snonaut3} we treat linear, non-autonomous 
parabolic equations in the classical Hilbert space setting as in 
\cite[Section~XVIII.3]{DLen5}.
We show that the exponent of time
integrability extrapolates from $2$ to $r \neq 2$ without losing maximal parabolic regularity 
and provide quantitative estimates on the size of $r$ in Theorem~\ref{thm-e-Hilbert}. 
A detailed motivation for these type of results is given at the end of Section \ref{Snonaut3}.
In Section~\ref{Snonaut5} we introduce the precise setting of the elliptic 
differential operators in divergence form with mixed boundary conditions
which we use in the remainder of the paper.
In Section~\ref{Snonaut6} we use recent results on elliptic regularity \cite{HJKR}
and autonomous parabolic regularity \cite{ABHR}, \cite{EHT} to 
obtain maximal parabolic regularity for autonomous operators in the $W^{-1,q}$-scale,
even for some $q < 2$.
Next in Section~\ref{Snonaut7} we exploit \eqref{e-interpolident0} in order to 
achieve maximal parabolic regularity for non-autonomous operators in the $W^{-1,q}$-scale.
To be precise, in Theorem \ref{t-main}
we show that non-autonomous operators $\frac{\partial }{\partial t}
-\nabla \cdot (\mu_t \nabla)$ satisfy maximal parabolic 
$L^r(J;W^{-1,q}_\fD(\Omega))$-regularity if $r \in (2-r_0,2+r_0)$ and  
$q \in (2-\varepsilon,2+\varepsilon)$.
The coefficient function $t \mapsto \mu_t$ can be discontinuous in time and space.
As in the preceding articles \cite{HJKR} and \cite{ABHR},
the geometric setting for the domains and boundary parts is 
extremely wide: the domains may even fail to be Lipschitz and the `Dirichlet'
boundary part $\fD$ is only required to be Ahlfors--David regular. 
In Section~\ref{Snonaut00}, the main results from previous sections are applied to related 
quasilinear problems.
In Section~\ref{Snonaut8}, we give an example of 
a non-autonomous parabolic operator with discontinuous coefficients
in both space and time to which Theorem~\ref{t-main} applies.

\section{Preliminaries} \label{Snonaut2}

Throughout this paper let $T > 0$ and set $J = (0,T)$.
Let us start by introducing the following (standard) definition.

\begin{defn} \label{d-bochner}
If $X$ is a Banach space and $r \in (1,\infty)$, 
then we denote by $L^r(J;X)$ the space of $X$-valued functions
$f$ on $J$ which are Bochner-measurable and for which $\int_J\|f(t)\|_X^r\,dt$
is finite.
We define $W^{1,r}(J;X):=\{u \in L^r(J;X): u' \in L^r(J;X)\}$,
where $u'$ is to be understood as the time derivative of $u$ in the sense of
 $X$-valued distributions (cf.\ \cite[Section~III.1]{Ama2}).
Moreover, we introduce the subspace 
\[
W_0^{1,r}(J;X):= \{ \psi \in W^{1,r}(J;X) : \psi(0)=0\}
 .  \]
We equip this subspace always with the norm $v \mapsto \|v'\|_{L^r(J;X)}$.
\end{defn}

In this paper we consider the following 
notion of maximal parabolic regularity in the non-autonomous case. 

\begin{defn} \label{d-maxpar}
Let $X$, $D$ be Banach spaces with $D$ densely and continuously embedded in $X$.
Let $J \ni t \mapsto \ca(t) \in \cl(D;X)$ be a bounded and measurable map
and suppose that the operator $\ca(t)$ is closed in $X$ for all $t \in J$.
Let $r \in (1,\infty)$.
Then we say that the family $\{\ca(t)\}_{t \in J}$ satisfies {\bf (non-autonomous) maximal parabolic
$L^r(J;X)$-regularity}, if for any $f \in L^r(J;X)$ there is a unique function
$u \in L^r(J;D) \cap W_0^{1,r}(J;X)$ which satisfies
\begin{equation} \label{e-0paragleich}
u'(t) +\ca(t)u(t)=f(t)
\end{equation}
for almost all $t \in J$.
We write
\[
\mr^r_0(J;D,X) := L^r(J;D) \cap W_0^{1,r}(J;X)
\]
for the space of maximal parabolic regularity. 
The norm of $u \in \mr^r_0(J;D,X)$ is
\[
\|u\|_{\mr^r_0(J;D,X)}
= \|u\|_{L^r(J;D)} + \|u'\|_{L^r(J;X)}
 .  \]
Then $\mr^r_0(J;D,X)$ is a Banach space.
\end{defn}

If all operators $\ca(t)$ are equal to one (fixed) operator $\ca_0$, 
and there exists an $r \in (1,\infty)$ such that 
$\{\ca(t)\}_{t \in J}$ satisfies maximal parabolic $L^r(J;X)$-regularity, 
then $\{\ca(t)\}_{t \in J}$ satisfies maximal parabolic $L^s(J;X)$-regularity
for all $s \in (1,\infty)$ and we say that 
$\ca_0$ satisfies {\bf maximal parabolic regularity on $X$}.
In all what follows, we denote the mapping which assigns to the right hand side 
$f \in L^r(J;X)$ in \eqref{e-0paragleich} the solution $u \in \mr^r_0(J;D,X)$ 
by $\bigl (\frac{\partial}{\partial t}+ \ca(\cdot)\bigr )^{-1}$
and $\bigl (\frac{\partial}{\partial t}+ \ca_0\bigr )^{-1}$,
respectively.

The $L^p$-spaces satisfy optimal interpolation properties with 
respect to the complex interpolation method.

\begin{prop} \label{p-interLpp}
Let $X,Y$ be two Banach spaces which form an interpolation couple.
Further, let $r_0, r_1 \in [1, \infty)$, $\theta \in (0,1)$ and set 
$r=\Bigl (\frac{1-\theta}{r_0} +\frac{\theta}{r_1}\Bigr )^{-1}$.
Then
\[
[L^{r_0}(J;X),L^{r_1}(J;Y)]_\theta =L^r(J;[X,Y]_\theta) 
\]
with {\bf equality} of norms.
\end{prop}
\begin{proof}
See \cite[Theorem~5.1.2]{BL}.
\end{proof}

We continue by quoting Sneiberg's extrapolation principle.

\begin{thm} \label{t-sneib}
Let $F_1, F_2,Z_1, Z_2$ be Banach spaces such that $(F_1, F_2)$ and $(Z_1, Z_2)$ 
are interpolation couples.
Assume $\mathcal R \in \cl(F_1;Z_1) 
\cap  \cl(F_2;Z_2)$ and put 
\[
\gamma:=\max\bigl (\|\mathcal R\|_{F_1 \to Z_1}, 
\|\mathcal R\|_{F_2 \to Z_2}\bigr ).
\]
 Suppose that for one $\theta \in (0,1)$ the operator
$\mathcal R \colon [F_1,F_2]_\theta \to [Z_1,Z_2]_\theta $ is a topological isomorphism
and let $\beta\ge \|\mathcal R^{-1}\|_{[Z_1,Z_2]_\theta \to [F_1,F_2]_\theta}$.
Then one has the following.
\begin{tabel} 
\item \label{p-sneib-1}
If $\tilde \theta \in (0,1)$ and 
\[
|\theta-\tilde\theta| < \frac{\min(\theta, 1-\theta)}{1+\beta \gamma} ,
\]
then 
$\mathcal R \colon [F_1,F_2]_{\tilde\theta} \to [Z_1,Z_2]_{\tilde\theta} $ 
remains surjective.
\item \label{p-sneib-2}
If $\tilde \theta \in (0,1)$ and 
\[
|\theta-\tilde\theta| \leq \frac{1}{6}\,\frac{ \min(\theta, 1-\theta)}{1 +2 \beta \gamma},
\]
then
$\mathcal R \colon [F_1,F_2]_{\tilde\theta} \to [Z_1,Z_2]_{\tilde\theta}$ remains an isomorphism
and
\[
\|\mathcal R^{-1} \|_{[Z_1,Z_2]_{\tilde\theta} \to [F_1,F_2]_{\tilde\theta}} 
\le 8 \beta .
\]
\end{tabel}
\end{thm}

\begin{proof} 
Essentially, the theorem was discovered by Sneiberg \cite{Sneiberg} and elaborated in more
detail in \cite{VignatiVignati}. 
The explicit quantitative estimates as quoted here were worked out
only recently in \cite[Propositions~1.3.27 and 1.3.25]{Egert}.
\end{proof}

\section{Pairs of common maximal parabolic regularity}  \label{Snonaut4}

The aim of this section is to provide an abstract setting in which the property of 
non-autonomous maximal parabolic regularity can be extrapolated by using Sneiberg's theorem. 
\begin{defn} \label{d-commonpair}
Let $X,Y,D,E$ be Banach spaces, with continuous embeddings $X \hookrightarrow Y$, 
$D \hookrightarrow X$ and $E \hookrightarrow Y$.
Suppose that $D$ is dense in $X$ and $E$ is dense in $Y$.
Then we say that 
the tuples $(D,X)$ and $(E,Y)$ form a 
{\bf pair of common maximal parabolic regularity} if there is an operator 
$\cb \in \cl(E;Y)$  such that  
$D = \{ x \in E \cap X : \cb x \in X \} $, 
the operator $\cb$ satisfies maximal parabolic regularity on $Y$,  and
the restriction $\cb|_D$ satisfies $\cb|_D \in \cl(D;X)$ and maximal parabolic regularity on $X$.
\end{defn}

For convenience, as is common for interpolation results, 
we sometimes extend the notation $\cb$ to $\cb|_D$ or other restrictions of $\cb$ in this section.

If $(D,X)$ and $(E,Y)$ form a pair of common maximal parabolic regularity, then 
the spaces of maximal parabolic regularity interpolate as follows.

\begin{lemma} \label{l-inter}
Assume that $(D,X)$ and $(E,Y)$ form a pair of common maximal parabolic regularity.
Let $r_0,r_1 \in [1,\infty)$.
Then 
\begin{equation} \label{e-innt}
[\mr^{r_0}_0(J;D,X)), \mr^{r_1}_0(J;E,Y)]_\theta 
=
\mr^r_0(J;[ D, E]_\theta, [X,Y]_\theta)
\end{equation}
for all $\theta \in (0,1)$, where
$\frac{1}{r} =\frac{1-\theta}{r_0}+\frac{\theta}{r_1}$.
\end{lemma}
\begin{proof}
The pair $\bigl (\mr^{r_0}_0(J;D,X)), \mr^{r_1}_0(J;E,Y)\bigr )$
is an interpolation couple since both spaces continuously inject into
$\mr^{\min(r_0,r_1)}_0(J; E, Y)$.
It follows from \cite[Lemma 5.3]{HaR} that $\cb$ satisfies maximal parabolic 
regularity on $[X,Y]_\theta$, 
so
\begin{align}\label{iso1}
 \frac{\partial }{\partial t} +\cb \colon 
\mr^r_0(J;[ D, E]_\theta,[X,Y]_\theta) \to L^r(J;[X,Y]_\theta) 
\end{align}
is an isomorphism.
Here we used that $(D,X)$ and $(E,Y)$ form a pair of common maximal parabolic regularity.

On the other hand, $\frac{\partial }{\partial t} +\cb$ is an isomorphism from 
$\mr^{r_0}_0(J; D,X)$ onto $L^{r_0}(J;X)$ and 
from $\mr^{r_1}_0(J; E, Y)$ onto $L^{r_1}(J;Y)$.
Hence by interpolation, the operator 
\begin{align}\label{iso2} 
\frac{\partial }{\partial t} +\cb \colon 
[\mr^{r_0}_0(J;D,X), \mr^{r_1}_0(J;E,Y)]_\theta  
 \to [L^{r_0}(J;X), L^{r_1}(J;Y)]_\theta,  
\end{align}
is an isomorphism. 
The statement follows from combining \eqref{iso1}, \eqref{iso2} and Proposition~\ref{p-interLpp}. 
\end{proof}

This interpolation result together with Theorem~\ref{t-sneib} enables us to extrapolate
maximal parabolic regularity to non-autonomous parabolic operators.
We need a simple lemma.

\begin{lem} \label{l-measurable}
Let $X,Y,D,E$ be Banach spaces, with continuous embeddings $X \hookrightarrow Y$, 
$D \hookrightarrow X$ and $E \hookrightarrow Y$.
Suppose $D$ is dense in $X$ and $E$ is dense in $Y$.
Let  $\{ \cb(t)\}_{t\in J}$ be a subset of $\cl(E;Y)$ satisfying 
$\sup_{t \in J} \|\cb(t)\|_{E \to Y} < \infty$.
Assume $\cb(t)|_D \in \cl(D;X)$  for all $t \in J$.
Moreover, suppose that 
$\sup_{t \in J} \|\cb(t)\|_{D \to X} < \infty$.
Let $\theta \in (0,1)$.
Then one has the following.
\begin{tabel} 
\item \label{l-measurable-1.5}
If $t \in J$, then $\cb(t)|_{[D,E]_\theta}\in\cl([D,E]_\theta; [X,Y]_\theta)$.
Moreover 
\begin{equation} \label{e-supremuminterp}
\sup_{t \in J} \|\cb(t)\|_{[D;E]_\theta \to [X;Y]_\theta} < \infty.
\end{equation}
\item \label{l-measurable-2}
Suppose $J \ni t \mapsto \cb(t)|_D $ is measurable.
Then the map 
$J \ni t \mapsto \cb(t)\psi \in [X,Y]_\theta$ is measurable
for all $\psi \in [D,E]_\theta$.
\item \label{l-measurable-3}
Let $r \in (1,\infty)$.
Then the map
\begin{equation} \label{e-isoOO}
\frac{\partial}{\partial t} +\cb(\cdot) 
   \colon \mr^r_0(J;[D,E]_\theta,[X,Y]_\theta) \to L^r(J;[X,Y]_\theta)
\end{equation} 
is well defined.
\end{tabel}
\end{lem}
\begin{proof}
`\ref{l-measurable-1.5}'.
This is well known from complex interpolation.

`\ref{l-measurable-2}'.
By assumption, for every $\psi \in D$ the map $J \ni t \mapsto \cb(t)\psi \in X$ 
is measurable.
Since the inclusion $X \hookrightarrow [X,Y]_\theta$ is 
continuous, also the map $J \ni t \mapsto  \cb(t)\psi \in [X,Y]_\theta$ is measurable
for all $\psi \in D$.
Let now $\psi \in [D,E]_\theta$.
Because $D$ is dense in $[D,E]_\theta$
by \cite[Theorem~1.9.3]{Tri}, there
is a sequence $\{\psi_n\}_n$ in $D$ which converges to $\psi$ in $[D,E]_\theta$.
But then
\eqref{e-supremuminterp} implies that the function 
$J \ni t \mapsto \cb(t)\psi \in [X,Y]_\theta$ is the pointwise limit of the functions 
$J \ni t \mapsto \cb(t)\psi_n \in [X,Y]_\theta$.
Hence it is measurable.

`\ref{l-measurable-3}'.
One can easily deduce from Statement~\ref{l-measurable-2} that for every 
$\eta \in (0,1)$ and $v \in L^r(J;[D,E]_\eta)$
the map
\[
J \ni t \mapsto \cb(t)v(t) \in [X,Y]_\eta
\]
is measurable.
Since $\sup_{t \in J} \|\cb(t)\|_{[D;E]_\theta \to [X;Y]_\theta} < \infty$,
it follows that \eqref{e-isoOO} is well defined.
\end{proof}

The main theorem of this section is the following.

\begin{thm} \label{t-extrapolatabstr}
Suppose the tuples $(D,X)$ and $(E,Y)$ form a pair of common maximal parabolic regularity.
For all $t \in J$ let $\cb(t) \in \cl(E;Y)$ and suppose that $\cb(t)|_D \in \cl(D;X)$.
Assume that the maps
\[
J \ni t \mapsto \cb(t)|_D \in \cl(D;X) 
\quad \mbox{and} \quad
J \ni t \mapsto \cb(t) \in \cl(E;Y) 
\]
are measurable and
$\sup_{t \in J} ( \|\cb(t)\|_{D \to X} +\|\cb(t)\|_{E \to Y} ) < \infty$.
Let $r_0, r_1 \in (1,\infty)$ and $\theta \in (0,1)$.
Set $\frac{1}{r} = \frac{1-\theta}{r_0} + \frac{\theta}{r_1}$.
Suppose that 
\begin{equation} \label{e-iso;1}
\frac{\partial}{\partial t} +\cb(\cdot) 
   \colon \mr^r_0(J;[D,E]_\theta,[X,Y]_\theta) \to L^r(J;[X,Y]_\theta)
\end{equation} 
 is a topological isomorphism.
Then there exists an $\varepsilon \in (0,\min(\theta,1-\theta))$ such that 
\[
\frac{\partial}{\partial t} +\cb(\cdot) 
   \colon \mr^s_0(J;[D,E]_{\tilde \theta},[X,Y]_{\tilde \theta}) \to L^s(J;[X,Y]_{\tilde \theta})
\]
is a topological isomorphism for all $\tilde \theta \in (\theta - \varepsilon,\theta + \varepsilon)$,
where $s:= \bigl (\frac{1-\tilde\theta}{r_0}+ \frac{\tilde\theta}{r_1}\Big )^{-1}$
\end{thm}
\begin{proof}
The operators
\[
\frac{\partial}{\partial t} +\cb(\cdot) \colon \mr^{r_0}_0(J;D,X) \to L^{r_0}(J;X)
\]
and 
\[
\frac{\partial}{\partial t} +\cb(\cdot) \colon \mr^{r_1}_0(J;E,Y) \to L^{r_1}(J;Y)
\]
are continuous and their norms are bounded by $\max(1,\gamma)$, where 
$\gamma = \sup_{t \in J} \|\cb(t)\|_{D \to X} + \sup_{t \in J} \|\cb(t)\|_{E \to Y}$.
Moreover, since the tuples $(D,X)$ and $(E,Y)$ form a pair of 
common maximal parabolic regularity by assumption, we may apply Lemma~\ref{l-inter}
to obtain the interpolation identity \eqref{e-innt}.
Using also Proposition~\ref{p-interLpp}, one can rewrite \eqref{e-iso;1}
as a topological isomorphism 
\begin{equation}
\frac{\partial}{\partial t} +\cb(\cdot)
\colon [\mr^{r_0}_0(J;D,X), \mr^{r_1}_0(J;E,Y)]_\theta 
   \to [L^{r_0}(J;X),L^{r_1}(J;Y)]_\theta
 .  
\label{et-extrapolatabstr;1}
\end{equation}
By Theorem~\ref{t-sneib} the isomorphism in \eqref{et-extrapolatabstr;1}
remains a topological isomorphism,
if $\theta$ is replaced by $\tilde\theta \in (0,1)$, $r$ is replaced by 
$s = \bigl (\frac{1-\tilde\theta}{r_0}+ \frac{\tilde\theta}{r_1}\Big )^{-1}$
and $\tilde\theta$ is sufficiently close to $\theta$.
Then the theorem follows by applying again 
Lemma~\ref{l-inter} and Proposition~\ref{p-interLpp}.
\end{proof}

As a simple consequence to Theorem~\ref{t-extrapolatabstr}, we obtain the following. 

\begin{corollary} \label{c-extrap}
Let $X$, $D$ be Banach spaces with $D$ densely and continuously embedded in $X$.
Let $J \ni t \mapsto \cb(t) \in \cl(D;X)$ be a bounded and measurable map
and suppose that the operator $\cb(t)$ is closed in $X$ for all $t \in J$.
Let $\ck \in \cl(D;X)$ and suppose that $\ck$ satisfies maximal parabolic regularity on $X$.
Let $r \in (1,\infty)$.
Suppose that $\{\cb(t)\}_{t\in J}$ satisfies  maximal parabolic 
$L^r(J;X)$-regularity.
Then there exists an open interval $I \subset (1,\infty)$ with 
$r\in I$ such that $\{\cb(t)\}_{t\in J}$ satisfies  maximal parabolic 
$L^s(J;X)$-regularity for all $s\in I$.
\end{corollary}
\begin{proof}
Choose $D=E$ and $X=Y$ in Lemma \ref{l-inter} and Theorem \ref{t-extrapolatabstr}.
\end{proof}

Theorem~\ref{t-extrapolatabstr} and Corollary~\ref{c-extrap} show the abstract
 principle we use in the sequel for the extrapolation of maximal $L^r(J;X)$-regularity. 
In Corollary~\ref{c-extrap} the space $D$ and $X$ are connected via \emph{some} 
autonomous reference operator $\ck$ that has maximal regularity.
Then Corollary~\ref{c-extrap} gives that for every 
non-autonomous operator family $J \ni t \mapsto \cb(t) \in \cl(D;X)$
with maximal parabolic $L^r(J;X)$-regularity the regularity extrapolates
around~$r$.
We expect that the interpolation formula 
\eqref{e-innt} is of independent interest and may serve for other purpose also in
different contexts.
In Section~\ref{Snonaut3} we apply this principle to non-autonomous families
 of operators in Hilbert spaces generated by families of sesquilinear forms, 
and in Sections \ref{Snonaut5}--\ref{Snonaut7}
to non-autonomous elliptic differential operators in Sobolev spaces.
In Theorem~\ref{t-extrapolatabstr} and Corollary~\ref{c-extrap}, however, the quantitative 
estimates of Theorem~\ref{t-sneib} are lost, as \eqref{e-innt} holds only with 
equivalence of norms and the exact constants seem to be hard to control 
(cf.\ Remark \ref{r-Pich}).
In order to get some quantitative results in the specific
 settings of Sections \ref{Snonaut3} and \ref{Snonaut7}, we will use direct 
proofs based on Theorem~\ref{t-sneib}.
The following definition and interpolation
 result will be useful tools. 

\begin{defn} \label{d-tilde}
Let $X$, $D$ be Banach spaces with $D$ densely and continuously embedded in $X$.
Let $\ck \in \cl(D;X)$ and suppose that $\ck$ satisfies maximal parabolic regularity on $X$.
For all $r \in (1,\infty)$, denote by $\mr^r_0(J;D,X)\,\widetilde{\;}\;$ 
the space $\mr^r_0(J;
D,X)$ equipped with the  norm
\[
u \mapsto 
\|u\|_{\mr^r_0(J;D,X)\,\widetilde{\;}}
= \|(\frac{\partial }{\partial t} + \ck ) u \|_{L^r(J;X)} .
\]
\end{defn}
It will be clear from the context which operator $\ck$ is involved.
Obviously
\begin{equation}\label{e-iso}
\frac{\partial }{\partial t} + \ck \colon \mr^r_0(J;D,X)\,\widetilde{\;}\; \to L^r(J;X)
\end{equation}
is an \emph{isometric} isomorphism.
For all $r \in (1,\infty)$ define 
\begin{equation}\label{d-C}
C_{\ck}^r := \|(\frac{\partial }{\partial t} + \ck )^{-1} \|_{L^r(J;X) \to \mr^r_0(J;D,X)} .
\end{equation}
Then 
\begin{equation}
\|u\|_{\mr^r_0(J;D,X)}
\leq C_{\ck}^r \, \|(\frac{\partial }{\partial t} + \ck ) u\|_{L^r(J;X)}
= C_{\ck}^r \, \|u\|_{\mr^r_0(J;D,X)\,\widetilde{\;}}
\label{ed-tilde;1}
\end{equation}
and 
\begin{eqnarray}
\|u\|_{\mr^r_0(J;D,X)\,\widetilde{\;}}
& = & \|(\frac{\partial }{\partial t} + \ck ) u\|_{L^r(J;X)}  \nonumber  \\
& \leq & \|u'\|_{L^r(J;X)} + \|\ck\|_{D \to X} \, \|u\|_{\mr^r_0(J;D,X)}  \nonumber  \\
& \leq & (1 \vee \|\ck\|_{D \to X}) \, \|u\|_{\mr^r_0(J;D,X)}
\label{ed-tilde;2}
\end{eqnarray}
for all $u \in \mr^r_0(J;D,X)$.
So the two norms are equivalent.

\begin{lem} \label{lemma-t-intgeropo}
Adopt the notation as in Definition \ref{d-tilde}.
Let $\theta \in (0,1)$ and $r_0, r_1 \in (1,\infty)$.
Then
\begin{equation} \label{e-interrrpol}
[\mr^{r_0}_0(J;D,X)\,\widetilde{\;}, 
 \mr^{r_1}_0(J;D,X)\,\widetilde{\;}\,]_\theta
= \mr^r_0(J;D,X)\,\widetilde{\;}
\end{equation}
with {\bf equality} of norms, where $r \in (1,\infty)$ is such that
$\frac{1}{r} = \frac{1-\theta}{r_0}+\frac{\theta}{r_1}$.
\end{lem}
\begin{proof}
By Proposition~\ref{p-interLpp} one has the interpolation identity
\begin{equation} \label{e-interrrpol0}
{}[ L^{r_0}(J;X),L^{r_1}(J;X)]_\theta=L^r(J;X) ,
\end{equation}
with equality of norms.
Then the \emph{isometric} isomorphism \eqref{e-iso} 
carries \eqref{e-interrrpol0} over to \eqref{e-interrrpol}, with equality of norms.
\end{proof}

\section{Maximal parabolic regularity and form methods}  \label{Snonaut3}
In this section, we will investigate maximal parabolic regularity in a Hilbert
space setting. 
For convenience, recall the following classical existence result,
which will serve as the starting point for the remainder of this paper.

\begin{prop} \label{p-Lions}
Let $V$ and $H$ be Hilbert spaces with $V$ densely and continuously
embedded into $H$.
For every $t \in J$ let $\mathfrak s_t$ be a sesquilinear form on $V$.
Let $c_\bullet, c^\bullet > 0$.
Suppose that the map $t \mapsto \mathfrak s_t[\varphi, \psi]$ 
from $J$ into $\Ci$ is measurable for all 
$\psi, \varphi \in V$.
Suppose that $\RRe \mathfrak s_t[\psi,\psi]  \ge c_\bullet\|\psi\|^2_V$ 
and $|\mathfrak s_t[\psi, \varphi]| \le c^\bullet \, \|\psi\|_V \, \|\varphi\|_V$
for all $\varphi,\psi \in V$ and $t \in J$.
For all $t \in J$ 
let $\cb(t) \colon V \to V^*$ be the linear operator which 
is induced by the sesquilinear form $\mathfrak s_t$.
Then for all $f \in L^2(J;V^*)$ there exists a unique 
$u \in \mr^2_0(J;V,V^*)$ such that 
\begin{equation} \label{e-nonautpara}
u'(t)+\cb(t)u(t) =f(t)
\end{equation}
for a.e.\ $t \in J$.
Moreover, 
\begin{equation} \label{e-apriori1}
\|u\|_{L^2(J;V)} \le \frac{1}{c_\bullet} \|f\|_{L^2(J;V^*)}
\quad \mbox{and} \quad
\|u'\|_{L^2(J;V^*)} \le \bigl (1+\frac{c^\bullet}{c_\bullet} \bigr )\|f\|_{L^2(J;V^*)}.
\end{equation}
In particular,
\begin{equation}
\|u\|_{\mr^2_0(J;V,V^*)}
\leq \frac{1 + c_\bullet + c^\bullet}{c_\bullet} \, \|f\|_{L^2(J;V^*)}
\label{ep-Lions;10}
\end{equation}
\end{prop}
\begin{proof}
The existence and uniqueness in (\ref{e-nonautpara}) follows from 
\cite[Section~XVIII.3 Remark 9]{DLen5}.
We next prove the estimates (\ref{e-apriori1}).
If $\tau \in \overline J$, then the energy equality 
\[
\frac{1}{2} \|u(\tau)\|_H^2 + \RRe \int_0^\tau  \mathfrak s_t[u(t),u(t)] \, dt
=\RRe \int_0^\tau \langle f(t),u(t) \rangle_{V^* \times V} \, dt
\]
follows from \eqref{e-nonautpara}, cf.\ \cite[Section~XVIII.3 Equation~3.86]{DLen5}.
Using the uniform coercivity,
this gives
\begin{eqnarray*}
c_\bullet \int_0^T \|u(t)\|^2_V \, dt 
& \le & \RRe \int_0^T  \mathfrak s_t[u(t),u(t)] \, dt \\
& \le & \RRe \int_0^T \langle f(t),u(t) \rangle _{V^* \times V} \, dt  \\
& \leq & \int_0^T \|f(t)\|_{V^*} \, \|u(t)\|_V \,  dt 
\le \|f\|_{L^2(J;V^*)} \, \|u\|_{L^2(J;V)}, 
\end{eqnarray*}
which proves the first inequality in \eqref{e-apriori1}.
Note that 
$\|\cb(t)\|_{V \to V^*}\le c^\bullet$ for all $t \in J$.
Therefore
\begin{eqnarray*}
\|u'\|_{L^2(J;V^*)} 
& \le & \|f\|_{L^2(J;V^*)} +\|\cb(\cdot)u(\cdot)\|_{L^2(J;V^*)} \\
& \le & \|f\|_{L^2(J;V^*)} +c^\bullet \|u\|_{L^2(J;V)} 
\le \bigl (1+\frac{c^\bullet}{c_\bullet} \bigr )\|f\|_{L^2(J;V^*)}
\end{eqnarray*}
and the second inequality in \eqref{e-apriori1} follows.
\end{proof}

Adopt the notation and assumptions of Proposition \ref{p-Lions}.
We are interested in the problem for which $r \in (1,\infty) \setminus \{ 2 \} $ 
the map
\begin{equation} \label{e-surjecHilbert}
\frac{\partial}{\partial t} +\cb(\cdot) \colon \mr^r_0(J;V,V^*) \to L^r(J;V^*) 
\end{equation}
is still an isomorphism.

Let $\cj  \colon V \to V^*$ be the duality map obtained by the 
Riesz representation theorem.
The defining relation is
\begin{equation} \label{e-dualitymap1}
\langle  \cj \psi,\varphi \rangle_{V^* \times V} 
= (\psi,\varphi)_V 
\end{equation}
for all $\psi, \varphi \in V$.
It follows from \eqref{e-dualitymap1} 
that the operator $\cj  \colon V \to V^*$ is the operator 
associated with the sesquilinear form which is the 
scalar product in $V$.
Note that the sesquilinear form is bounded, with constant $1$, and has coercivity constant 
which is also equal to~$1$.
One deduces from Proposition~\ref{p-Lions}
that for all $f \in L^2(J;V^*)$
the equation
\[
u' +\cj  u =f
\]
admits exactly one solution $u \in \mr^2_0(J;V,V^*)$.
Consequently, the operator $\cj $ satisfies maximal parabolic regularity on $V^*$.
Therefore we can apply Corollary~\ref{c-extrap} with $\ck = \cj$.
It follows for the operator family $ \{ \cb(t) \} _{t \in J}$ in Proposition \ref{p-Lions} that 
there is an open interval $I \ni 2$ 
such that the map in \eqref{e-surjecHilbert} is still an isomorphism for all $r\in I$. 
Using Lemma \ref{lemma-t-intgeropo} and the Sneiberg theorem
we also prove a quantitative result on $I$.
The main theorem of this section is the following.

\begin{thm} \label{thm-e-Hilbert}
Adopt the notation and assumptions of Propositions~\ref{p-Lions}.
Let $r_0 \in (2,\infty)$, $r_1 \in (1,2)$
and define $\theta \in (0,1)$ such that 
$\frac{1}{2}=\frac{1-\theta}{r_0}+\frac{\theta}{r_1}$.
Set 
\[
C_{\cj} =\max_{j \in \{ 0,1 \} }\|\bigl ( \frac{\partial }{\partial t}
 +\cj  \bigr )^{-1}\|_{L^{r_j}(J;V^*) \to \mr^{r_j}_0(J;V,V^*)}
 .  \]
Let $\tilde \theta \in (0,1)$.
Let $r \in (1,\infty)$ be such that 
$\frac{1}{r}= \frac{1-\tilde\theta}{r_0}+\frac{\tilde\theta}{r_1}$.
Then one has the following.
\begin{tabel}
\item  \label{thm-0001-1}
If
\[
|\theta -\tilde\theta| 
< \frac{\min(\theta, 1-\theta)}{1+( 1+\frac{1+c^\bullet }{c_\bullet})\max(1,c^\bullet) C_{\cj}},
\]
then $\{\cb(t)\}_{t \in J}$
satisfies maximal $L^r(J;V^*)$-regularity.
\item  \label{thm-0001-2}
If
\[
|\theta-\tilde\theta| 
\leq \frac{1}{6} \, \frac{ \min(\theta, 1-\theta)}
                      {1 +2 ( 1+\frac{1+c^\bullet }{c_\bullet})\max(1,c^\bullet) C_{\cj}},
\]
then 
\[
\|\bigl (\frac{\partial}{\partial t} +\mathcal B(\cdot)\bigr )^{-1}
\|_{L^r(J;V^*) \to \mr^r_0(J;V,V^*)\,\widetilde{\;}} 
\le 8 \frac{1 + c_\bullet + c^\bullet}{c_\bullet} ,
\]
where the norm on $\mr^r_0(J;V,V^*)\,\widetilde{\;}$ is defined using the operator $\cj$.
\end{tabel}
\end{thm}

For the proof, we first show that injectivity is preserved. 

\begin{lem} \label{l-injective}
Adopt the notation and assumptions of Propositions~\ref{p-Lions}.
Then the
map 
\[
\frac{\partial}{\partial t} +\cb(\cdot) \colon \mr^r_0(J;V,V^*) \to L^r(J;V^*) 
\]
is injective for all $r \in (1,\infty)$.
\end{lem}
\begin{proof}
For all $t \in J$ the operator $\cb(t)$ is accretive on $V^*$.
Then the claim follows from \cite[Proposition~3.2]{ACFP}.
\end{proof}

\begin{proof}[{\bf Proof of Theorem~\ref{thm-e-Hilbert}}]
`\ref{thm-0001-1}'.
We apply Theorem~\ref{t-sneib} of Sneiberg with the spaces
$F_1=\mr^{r_0}_0(J;V,V^*)\,\widetilde{\;}$,
$F_2=\mr^{r_1}_0(J;V,V^*)\,\widetilde{\;}$,
$Z_1=L^{r_0}(J;V^*)$ and $Z_2=L^{r_1}(J;V^*)$. 
Then $[F_1,F_2]_\theta = \mr^2_0(J;V,V^*)\,\widetilde{\;}\;$
and $[Z_1,Z_2]_\theta = L^2(J;V^*)$, with equality of norms by 
Lemma~\ref{lemma-t-intgeropo} and Proposition~\ref{p-interLpp}.
Moreover, $\frac{\partial }{\partial t} + \cb(\cdot)$
is an isomorphism from $\mr^2_0(J;V,V^*)$ onto 
$L^2(J;V^*)$ by Proposition~\ref{p-Lions}.

We first estimate 
$\beta = \|\bigl (\frac{\partial }{\partial t} + \cb(\cdot)\bigr)^{-1}
\|_{[Z_1,Z_2]_\theta \to [F_1,F_2]_\theta}$.
Let  $u \in [F_1,F_2]_\theta$.
Then 
\begin{eqnarray*}
\|u\|_{[F_1,F_2]_\theta}
& = & \|u\|_{\mr^2_0(J;V,V^*)\,\widetilde{\;} }
\leq \|u\|_{\mr^2_0(J;V,V^*)}
\leq \Big( \frac{1+c_\bullet + c^\bullet }{c_\bullet} \Big) 
    \|\Big( \frac{\partial }{\partial t} + \cb(\cdot) \Big) u\|_{ L^2(J;V^*) } ,
\end{eqnarray*}
where we used (\ref{ed-tilde;2}) and (\ref{ep-Lions;10}) in the two inequalities.
So 
$\beta \leq \frac{1+c_\bullet + c^\bullet }{c_\bullet}$.

Next we estimate 
\[
\gamma:=\max_{j \in \{ 1,2 \} }
   \|\frac{\partial }{\partial t} + \cb(\cdot)\|
_{\mr^{r_j}_0(J;V,V^*) \,\widetilde{\;} \, \to L^{r_j}(J;V^*)} ,
\] 
which is the second input for the calculation of the admissible interpolation parameters in
Theorem~\ref{t-sneib}.
Let $j \in \{ 1,2 \} $ and $u \in \mr^{r_j}_0(J;V,V^*)$.
Then 
\[
\| \Big( \frac{\partial }{\partial t} + \cb(\cdot) \Big) u\|_{L^{r_j}(J;V^*)}
\leq (1 \vee c^\bullet) \, \|u\|_{\mr^{r_j}_0(J;V,V^*)}  
\leq C^{r_j}_\cj \, \|u\|_{ \mr^{r_j}_0(J;V,V^*)\,\widetilde{\;} }
,  \]
where we used (\ref{ed-tilde;1}) in the second step.
Hence
\[
\gamma 
\leq \max(1,c^\bullet) \, C_{\cj}.
\]
The operator 
\[
\frac{\partial}{\partial t} +\cb(\cdot) \colon \mr^r_0(J;V,V^*) \to L^r(J;V^*) 
\]
is obviously continuous, and, by
Lemma \ref{l-injective}, it is injective.
Moreover, it is surjective by Theorem~\ref{t-sneib}\ref{p-sneib-1} for the claimed interpolation parameters.
Then Theorem~\ref{thm-e-Hilbert} follows from the open
mapping theorem.

`\ref{thm-0001-2}'.
This is proved analogously by using Theorem~\ref{t-sneib}\ref{p-sneib-2}.
\end{proof}

\begin{rem}\label{r-Pich}
It would be very interesting to get upper and lower bounds on the
exponents $r$, dependent of the constants $c_\bullet$ and $c^\bullet$.
The difficulty is to get explicit estimates of $C_{\cj}^r$ dependent on $r$. 
For results related to this problem, we refer to
\cite{CannarsaVespri}.
\end{rem}

We continue with some motivation for Theorem~\ref{thm-e-Hilbert}.
In Proposition~\ref{p-Lions} one has $f \in L^2(J;V^*)$ and 
therefore $u \in \mr^2(J;V,V^*) \subset C(J;H)$.
So $u$ is continuous from $J$ into~$H$.
In many settings, it
turns out that $f \in L^r(J;V^*)$ for some $r >2$, because in 
many applications like spatial discretization, one is confronted with step functions. 
Then Theorem~\ref{thm-e-Hilbert} gives that $u \in \mr^r_0(J;V,V^*)$.
We show in the next proposition that then the function 
$u \colon J \to H$ is H\"older continuous.
Moreover, the orbits are relatively
 compact in $H$, if the embedding $V \hookrightarrow V^*$ is compact.
This additional regularity can be essential for studying related semi- and quasilinear problems
(cf.\ Section~\ref{Snonaut00}), large-time asymptotic behaviour (cf.\ \cite{Schn}), or for
optimal control problems with tracking type objective functional (cf.\ \cite{HMRR}).

\begin{prop} \label{pnonaut410}
Let $V$ and $H$ be Hilbert spaces with $V$ densely and continuously embedded into~$H$.
Then one has the following.
\begin{tabel}
\item \label{pnonaut410-1}
If $r \in (1,\infty)$ and $\theta \in (0,1-\frac{1}{r})$, then
$\mr^r_0(J;V,V^*) \hookrightarrow C(\overline J;(V^*,V)_{1-\frac{1}{r},r})$.
\item \label{pnonaut410-2}
If $r \in (1,\infty)$ and $\theta \in (0,1-\frac{1}{r})$, then
$\mr^r_0(J;V,V^*) \hookrightarrow C^\alpha(J;(V^*,V)_{\theta,1})$,
where $\alpha = 1-\frac{1}{r}-\theta$.
\item \label{pnonaut410-3}
If $r \in (2,\infty)$, then $\mr^r_0(J;V,V^*)\hookrightarrow C^\alpha(J;H)$ for all 
$\alpha \in (0,\frac{1}{2} - \frac{1}{r})$.
\item \label{pnonaut410-4}
If $r \in (2,\infty)$ and $V$ embeds compactly into $V^*$, then 
$\mr^r_0(J;V,V^*) \hookrightarrow C(\overline J;H)$ and the embedding is compact.
\end{tabel}
\end{prop}
\begin{proof}
`\ref{pnonaut410-1}'. 
See \cite[Theorem~III.4.10.2]{Ama2}.

`\ref{pnonaut410-2}'. 
See \cite[Theorem~3]{Ama3}, or \cite[Lemma~2.11(b)]{DisserterElstRehberg} for an elementary proof. 

`\ref{pnonaut410-3}'. 
Set $\theta = 1 - \frac{1}{r} - \alpha$. 
Then $\theta \in (\frac{1}{2},1)$.
It follows from Statement~\ref{pnonaut410-2} that 
$\mr^r_0(J;V,V^*) \hookrightarrow C^\alpha(J;(V^*,V)_{\theta,1})$.
Moreover, $(V^*,V)_{\theta,1} \hookrightarrow (V^*,V)_{\frac{1}{2},2} = [V^*,V]_\frac{1}{2} = H$.
Hence 
$\mr^r_0(J;V,V^*)\hookrightarrow C^\alpha(J;H)$.

`\ref{pnonaut410-4}'. 
Choose $\alpha \in (0,\frac{1}{2} - \frac{1}{r})$.
Set $\theta = 1 - \frac{1}{r} - \alpha$ as in the proof of Statement~\ref{pnonaut410-3}.
Since $V$ embeds compactly into $V^*$, 
the embedding $(V^*,V)_{\theta,1} \hookrightarrow (V^*,V)_{\frac{1}{2},2}$ is 
compact by \cite[Section~3.8]{BL}.
Then the statement follows by the Ascoli theorem for vector-valued functions,
see \cite[Section~III.3]{Lang2}.
\end{proof}

\section{Elliptic differential operators} \label{Snonaut5}

In the sequel we will apply the above results for the derivation of regularity for 
non-autonomous parabolic differential operators on Sobolev spaces.
We first introduce the underlying elliptic setting.

Throughout the rest of this paper we fix a bounded open set 
$\Omega \subset \R^d$, where $d \geq 2$.
Let $\fD$ be a 
closed subset of the boundary
$\partial \Omega$ (to be understood as the Dirichlet boundary part).
Regarding our geometric setting, we suppose the following general conditions.

\begin{assu} \label{assu-general}
\mbox{}
 \begin{tabeleq}
 \item \label{assu-general:i} 
For every $x \in \overline{\partial \Omega \setminus \fD}$ there
exists an open neighbourhood $\mathfrak U_x$ of $x$ and a bi-Lipschitz map
$\phi_x$ from $\mathfrak U_x$ onto the cube $K :=
{({-1},1)}^d$, such that the following three conditions are satisfied:
\begin{align*}
  \phi_x(x) &= 0, \nonumber \\
  \phi_x(\mathfrak U_x \cap \Omega) &= \{ x \in K : x_d < 0 \} ,  \\
  \phi_x(\mathfrak U_x \cap \partial \Omega) 
&= \{ x \in K : x_d = 0 \}.
\end{align*}
\item   \label{assu-general:ii}
The set $\fD$
satisfies the \emph{Ahlfors--David condition}, that is 
there are $c_0, c_1 > 0$ and $r_{AD} > 0$, such that 
\begin{equation} \label{e-ahlf}
  c_0 r^{d-1} \le \mathcal H_{d-1} (\fD \cap B(x,r) ) \le c_1
r^{d-1}
\end{equation}
for all $x \in \fD$ and $r \in {(0,r_{AD}]}$,
where $\mathcal H_{d-1}$ denotes the
$(d-1)$-dimensional Hausdorff measure and $B(x,r)$ denotes the ball
with centre $x$ and radius $r$. 
 \item   \label{assu-general:iii}
The set $\Omega$ is a $d$-set in the sense of Jonsson--Wallin \cite[Chapter~II]{JW}.
This condition is, however, {\bf not} needed if the coefficient function $\mu$ or 
$(\mu_t)_{t \in J}$ is hermitian, which we introduce below.
That is, if $\mu(x)$ is symmetric or $\mu_t(x)$ is 
symmetric for all $x \in \Omega$ and $t \in J$, then we do not need that 
$\Omega$ is a $d$-set.
\end{tabeleq}
\end{assu}

\begin{rem} \label{r-surfmeas}
\mbox{}
 \begin{tabel}
\item  \label{r-surfmeas-3}
We emphasize that the cases 
$\fD = \partial \Omega$ or $\fD = \emptyset$ are allowed.
 \item  \label{r-surfmeas-1}
Condition~\eqref{e-ahlf} means that $\fD$ is a $(d-1)$-set in the sense 
of Jonsson--Wallin \cite[Chapter~II]{JW}.
 \item \label{r-surfmeas:ii} 
On the set 
$\partial \Omega \cap \bigl( \bigcup_{x \in \overline{\partial \Omega \setminus \fD}}
       \, \mathfrak U_x \bigr)$ the measure $\mathcal H_{d-1}$ is equal to the surface measure
$\sigma$, which can be constructed via the bi-Lipschitz charts
$\phi_x$ around these boundary points, see
\cite[Section~3.3.4~C]{EvG} or \cite[Section~3]{HaR2}.
In
particular, \eqref{e-ahlf} assures that 
$\sigma \bigl(\fD \cap \bigl( \bigcup_{x \in \overline{\partial \Omega \setminus \fD}}
    \, \mathfrak U_x \bigr) \bigr) > 0$ if $\fD \neq \emptyset$ and 
$\fD \neq \partial \Omega$.
\item \label{r-surfmeas-5}
Condition~\ref{assu-general:iii} excludes \emph{outward} cusps, 
but \emph{inward} cusps are allowed.
This condition is only used in Proposition~\ref{p-3}.
\end{tabel}
\end{rem}

\begin{defn} \label{d-1}
For all $q \in [1,\infty)$ we define the space 
$W^{1,q}_\fD(\Omega)$ as the completion
of
\[
C^\infty_\fD(\Omega) 
:= \{ \psi|_\Omega : \psi \in C_c^\infty(\R^d) \mbox{ and }
     \supp(\psi) \cap \fD = \emptyset \}
\]
with respect to the (standard) norm $\psi \mapsto \bigl( \int_\Omega |\nabla \psi|^q +
|\psi|^q \bigr)^{1/q}$.
If $q \in (1,\infty)$ then the (anti-) dual of
this space will be denoted by $W^{-1,q'}_\fD(\Omega)$, where $1/q + 1/q' = 1$.
Here, the dual is to be understood with respect to the extended $L^2$ scalar
product, or, in other words, $W^{-1,q'}_\fD(\Omega)$ is the space of continuous
antilinear functionals on $W^{1,q}_\fD(\Omega)$.

Since the domain $\Omega$ is kept fixed, we omit the
symbol `$\Omega$' in the sequel if no confusion is possible.
For example, we write $W^{1,q}_\fD$ instead of
 $W^{1,q}_\fD(\Omega)$. 
\end{defn}

The spaces $W^{1,q}_\fD$ and the spaces
 $W^{-1,q}_\fD$  interpolate with respect to 
the complex interpolation functor.

\begin{lemma} \label{l-1} 
Adopt Assumption~\ref{assu-general}. 
Let $q_1,q_2 \in (1,\infty)$ and $\theta \in (0,1)$.
Then 
\begin{eqnarray*} 
[W^{1,q_1}_\fD,W^{1,q_2}_\fD]_\theta 
& = &W^{1,q}_\fD \quad \mbox{and}  \\
{} [W^{-1,q_1}_\fD,W^{-1,q_2}_\fD]_\theta 
& = & W^{-1,q}_\fD, 
\end{eqnarray*}
where $\frac{1}{q}=\frac{1-\theta}{q_1} +\frac{\theta}{q_2}$.
\end{lemma}
\begin{proof}
See \cite{HJKR} Theorem~3.3 and Corollary~3.4.
\end{proof}

\begin{rem} \label{r-fortsetz}
By \cite[Lemma~3.2]{ABHR} there exists an extension operator 
$\mathcal E \colon W^{1,q}_\fD(\Omega) \to W^{1,q}_\fD(\R^d)$,
which is universal in $q \in [1,\infty)$.
Therefore one has the usual Sobolev embeddings available, including compactness.
\end{rem}

We now turn to the definition of the elliptic divergence form operators that
will be investigated.
 
For all $c_\bullet,c^\bullet > 0$ we denote by $\ce(c_\bullet,c^\bullet)$ the 
set of all measurable $\mu \colon \Omega \to \Ri^{d \times d}$ such that 
\[
\RRe (\mu(x) \xi,\xi)_{\Ci^d} \ge c_\bullet |\xi|^2
\quad \mbox{and} \quad
\|\mu(x)\|_{\cl(\Ci^d)} \leq c^\bullet
\]
for all $\xi \in \C^d$ and almost all $x \in \Omega$.
Moreover, define
\[
\ce = \bigcup_{c_\bullet,c^\bullet > 0} \ce(c_\bullet,c^\bullet)
,  \]
the set of all elliptic coefficient functions.

\begin{defn} \label{d-2}
For all $q \in (1,\infty)$ and $\mu \in \ce$ define the operator 
$\ca_q = \ca_q(\mu) \colon  W^{1,q}_\fD \to W^{-1,q}_\fD$ 
by
\[
 \langle \ca_q(\mu) \psi, \varphi\rangle_{W^{-1,q}_\fD \times W^{1,q'}_\fD} 
:= \int_\Omega \mu \nabla \psi \cdot \overline {\nabla \varphi} ,
\]
where $\psi \in W^{1,q}_\fD$ and $\varphi \in W^{1,q'}_\fD$.
Then 
\[
\|\ca_q\|_{W^{1,q}_\fD \to W^{-1,q}_\fD} 
\le \|\mu\|_{L^\infty(\Omega ;\cl(\Ci^d))}
\]
for all $q \in (1,\infty)$ by H\"older's inequality.

Moreover, for all $q \in [2,\infty)$ define the operator $\widetilde \ca_q = \widetilde \ca_q(\mu)$ in 
$W^{-1,q}_\fD$ by 
\[
\dom(\widetilde \ca_q)
= \{ \psi \in W^{-1,q}_\fD \cap W^{1,2}_\fD : 
        \ca_2 \psi \in W^{-1,q}_\fD \}
\]
and $\widetilde \ca_q = \ca_2|_{\dom(\widetilde \ca_q)}$.

Define the sesquilinear form $\gots = \gots_\mu \colon W^{1,2}_\fD \times W^{1,2}_\fD \to \Ci$ by
\[
\gots(\psi,\varphi) 
= \int_\Omega \mu \nabla \psi \cdot \overline {\nabla \varphi} 
 .  \]
Then $\gots$ is coercive and continuous.
Let $A = A(\mu)$ be the $m$-sectorial operator in $L^2$ associated to~$\gots$.
So $\dom A = \{ \psi \in W^{1,2}_\fD : \ca_2 \psi \in L^2 \} $
and $A = \ca_2|_{\dom A}$.
Let $S$ be the semigroup generated by $-A$.
Then $S$ extends consistently to a contraction semigroup $S^{(q)}$ on $L^q$ for 
all $q \in [1,\infty)$.
Moreover, $S^{(q)}$ is a holomorphic semigroup on $L^q$ for 
all $q \in [1,\infty)$ by \cite{ERe1} Theorem~3.1.
For all $q \in [1,\infty)$ we denote by $-A_q = - A_q(\mu)$ the generator of $S^{(q)}$.
\end{defn}

Clearly if $p,q \in (1,\infty)$ and $p \leq q$, then 
$\ca_q = \ca_p|_{W^{1,q}_\fD}$.
Thus the graph of $\ca_p$ is an extension of the graph of $\ca_q$.
As a consequence, $W^{1,q}_\fD \subset \dom \widetilde \ca_q$ for all $q \in [2,\infty)$.

Similarly, the graph of $A_p$ is an extension of the graph of $A_q$ if
$p,q \in [1,\infty)$ and $p \leq q$.

\section{Maximal parabolic regularity for differential operators} \label{Snonaut6}

In this section we prove that there exists a $q_0 \in (1,2)$ such that 
the operator $\ca_q + I$ or $\widetilde \ca_q + I$ satisfies maximal parabolic 
regularity on the space $W^{1,q}_\fD$ for all $q \in (q_0,\infty)$.
A key tool is a recent solution of the Kato square root problem.
In order to determine $q_0$ we need a definition.

\begin{defn} \label{d-isoindex}
Let $\mu \in \ce$.
We call a number $q \in (1,\infty)$ an {\bf isomorphism index} for the coefficient
 function $\mu$ if 
\[
\ca_q +I \colon  W^{1,q}_\fD \to  W^{-1,q}_\fD 
\]
is a topological isomorphism.
We denote by $\fI_\mu$ the set of isomorphism indices for $\mu$.
Although the set $\fI_\mu$ also depends on the set 
$\fD$, we suppress the dependence of $\fD$ in the notation.
\end{defn}

If $q \in (1,\infty)$ and $\mu \in \ce$, then by duality one obviously has 
$q \in \fI_\mu$ if and only if $q' \in \fI_{\mu^T}$.
This allows to concentrate to all $q \in [2,\infty)$.

\begin{lemma} \label{lnonauto302}
Let $\mu \in \ce$ and $q \in (2,\infty)$.
Then one has the following.
\begin{tabel}
\item \label{lnonauto302-1}
$q \in \fI_\mu$ if and only if the operator $\ca_q + I$ is surjective.
\item \label{lnonauto302-2}
If $q \in \fI_\mu$, then $\dom \widetilde \ca_q = W^{1,q}_\fD$.
\end{tabel}
\end{lemma}
\begin{proof}
`\ref{lnonauto302-1}'.
Let $\psi \in \ker(\ca_q + I)$.
Then $\psi \in W^{1,q}_\fD \subset W^{1,2}_\fD$
and $\int_\Omega \mu \nabla \psi \cdot \overline {\nabla \varphi} + \int_\Omega \psi \, \overline \varphi = 0$
for all $\varphi \in C^\infty_\fD$. 
By continuity and density the latter is then also valid for all $\varphi \in W^{1,2}_\fD$, 
in particular for $\varphi = \psi$.
Since $\mu$ is elliptic one deduces that $\psi = 0$.
So the operator $\ca_q + I$ is injective for all $q \in [2,\infty)$.

`\ref{lnonauto302-2}'.
Let $\psi \in \dom \widetilde \ca_q$.
Then $\psi \in W^{-1,q}_\fD \cap W^{1,2}_\fD$ and $\ca_2 \psi \in W^{-1,q}_\fD$.
Since $\ca_q + I$ is surjective, there exists a $\varphi \in W^{1,q}_\fD$ such that 
$(\ca_q + I) \varphi = (\ca_2 +I) \psi$.
Then $\varphi \in W^{1,2}_\fD$ and $\ca_2 \varphi = \ca_q \varphi$.
So $(\ca_2 + I) \varphi = (\ca_2 +I) \psi$.
Since $(\ca_2 + I)$ is injective, one deduces that $\psi = \varphi \in W^{1,q}_\fD$.
So $\dom \widetilde \ca_q \subset W^{1,q}_\fD$.
The reverse inclusion is trivial.
\end{proof}

The next proposition states that the set $\fI_\mu$ is always 
non-empty and open.

\begin{prop} \label{p-2} 
Adopt Assumption~\ref{assu-general}.
Then for all $\mu \in \ce$
the set $\fI_\mu$ is an open interval which contains~2.
Moreover, for all $c_\bullet,c^\bullet > 0$ there
are $\varepsilon, \delta > 0$ such that 
$(2-\delta, 2+\varepsilon) \subset \fI_\mu$ for all $\mu \in \ce(c_\bullet,c^\bullet)$,
and, in addition, 
\[
\sup_{\mu \in \ce(c_\bullet,c^\bullet)} 
    \|(\ca_q(\mu) +I)^{-1}\|_{W^{-1,q}_\fD \to W^{1,q}_\fD} < \infty.
\]
\end{prop}
\begin{proof}
If follows from Lemma~\ref{l-1} that $\fI_\mu$ is connected.
Moreover, $2 \in \fI_\mu$
by the  Lax--Milgram theorem. 
For the other assertions, see \cite[Theorem~5.6 and Remark~5.7]{HJKR}. 
\end{proof}

Assumption~\ref{assu-general} implies that the Kato problem for the operator $A_q$ with 
real measurable coefficients and boundary conditions is solved on $L^q$ for all $q \in (1,2]$.

\begin{prop}\label{p-3}
Adopt Assumption~\ref{assu-general}.
Let $\mu \in \ce$ and $q \in (1,2]$.
Then $\dom A_q^{1/2} = W^{1,q}_\fD$.
Moreover, the operator $(A_q + I)^{1/2}$ is a topological isomorphism from 
$W^{1,q}_\fD$ onto $L^q$.
Hence its adjoint map $((A_q + I)^{1/2})'$ 
is a topological isomorphism from  $L^p$ onto $W^{-1,p}_\fD$
for all $p \in [2,\infty)$.
\end{prop}
\begin{proof}
The case $q = 2$ is proved in \cite[Main Theorem 4.1]{EHT}.
The general case is in \cite[Theorem~5.1]{ABHR}. 
\end{proof}

Let $\mu \in \ce$ and $q \in (1,\infty)$.
If $\dom A_q^{1/2} = W^{1,q}_\fD$, then it follows from the 
closed graph theorem that the operator $(A_q + I)^{1/2}$ is a topological isomorphism from 
$W^{1,q}_\fD$ onto $L^q$.
As a consequence we can split the problem whether $\ca_q + I$ is an 
isomorphism from $W^{1,q}_\fD$ onto $W^{-1,q}_\fD$ into two parts:
from $W^{1,q}_\fD$ into $L^q$ and from $L^q$ into $W^{-1,q}_\fD$.

\begin{thm} \label{thmc-001}
Adopt Assumption~\ref{assu-general}.
\mbox{}
\begin{tabel}
\item  \label{thmc-001-1}
Let $q \in [2,\infty)$ and $\mu \in \ce$.
Then $q \in \fI_\mu$ if and only if $\dom A_q^{1/2} = W^{1,q}_\fD$.
\item  \label{thmc-001-1.5}
Let $\mu \in \ce$ and $q \in \fI_\mu$.
Then 
\[
\dom A_q 
= \{ \psi \in W^{1,q}_\fD : \ca_q \psi \in L^q \}
\]
and $A_q \psi = \ca_q \psi$ for all $\psi \in \dom A_q$.
\item \label{thmc-001-2}
For all $c_\bullet,c^\bullet > 0$ there exists an $\varepsilon > 0$ such that 
$\dom A_q^{1/2} = W^{1,q}_\fD$ for all 
$q \in (1,2+\varepsilon)$ and $\mu \in \ce(c_\bullet,c^\bullet)$.
\end{tabel}
\end{thm}
\begin{proof}
`\ref{thmc-001-1}'. 
Suppose that $\dom A_q^{1/2} = W^{1,q}_\fD$.
Then $(A_q + I)^{1/2}$ is an isomorphism from $W^{1,q}_\fD$ onto $L^q$.
Write $p = q'$.
Then it follows from Proposition~\ref{p-3} that $\dom A_p(\mu^T)^{1/2} = W^{1,p}_\fD$
and the operator $(A_p(\mu^T) + I)^{1/2}$ is a topological isomorphism from 
$W^{1,p}_\fD$ onto $L^p$.
Let $\psi \in W^{1,q}_\fD$. 
We first show that 
\begin{equation}
\langle (A_q + I)^{1/2} \psi, (A_p(\mu^T) + I)^{1/2} u \rangle_{L^q \times L^p}
= \langle (\ca_q + I) \psi, u \rangle_{W^{-1,q}_\fD \times W^{1,p}_\fD}
\label{ethmc-001;1}
\end{equation}
for all $u \in W^{1,p}_\fD$.
Let $u \in \dom(A_2(\mu^T))$.
Then $u \in \dom(A_p(\mu^T))$ and 
\begin{eqnarray*}
\langle (A_q + I)^{1/2} \psi, (A_p(\mu^T) + I)^{1/2} u \rangle_{L^q \times L^p}
& = & \langle (A_2 + I)^{1/2} \psi, (A_2(\mu^T) + I)^{1/2} u \rangle_{L^2 \times L^2}  \\
& = & (\psi, (A_2(\mu^T) + I) u)_{L^2}  \\
& = & \sum_{k,l=1}^d \int_\Omega \mu_{kl} \, (\partial_k \psi) \, \overline{ \partial_l u} + (\psi, u)_{L^2}  \\
& = & \langle (\ca_q + I) \psi, u \rangle_{W^{-1,q}_\fD \times W^{1,p}_\fD}
 . 
\end{eqnarray*}
So (\ref{ethmc-001;1}) is valid for all $u \in \dom(A_2(\mu^T))$.
Clearly $u \mapsto \langle (A_q + I)^{1/2} \psi, (A_p(\mu^T) + I)^{1/2} u \rangle_{L^q \times L^p}$
is continuous from $\dom (A_p(\mu^T) + I)^{1/2}$ into $\Ci$.
Also $u \mapsto \langle (\ca_q + I) \psi, u \rangle_{W^{-1,q}_\fD \times W^{1,p}_\fD}$ 
is continuous from $W^{1,p}_\fD$ into $\Ci$ and hence 
from $\dom (A_p(\mu^T) + I)^{1/2}$ into $\Ci$.
Since $\dom A_2(\mu^T)$ is dense in $\dom A_p(\mu^T)$, and hence in 
$\dom (A_p(\mu^T) + I)^{1/2}$, it follows by continuity that 
(\ref{ethmc-001;1}) is valid for all $u \in \dom (A_p(\mu^T) + I)^{1/2} = W^{1,p}_\fD$.

Let $\varphi \in W^{-1,q}_\fD$.
By the last part of Proposition~\ref{p-3} there exists a $\tau \in L^q$ such that 
$\langle \tau, (A_p(\mu^T) + I)^{1/2} u \rangle_{L^q \times L^p}
= \langle \varphi, u \rangle_{W^{-1,q}_\fD \times W^{1,p}_\fD}$ 
for all $u \in W^{1,p}_\fD$.
By assumption there exists a $\psi \in W^{1,q}_\fD$ such that 
$(A_q + I)^{1/2} \psi = \tau$.
Then 
\begin{eqnarray*}
\langle \varphi, u \rangle_{W^{-1,q}_\fD \times W^{1,p}_\fD}
& = & \langle \tau, (A_p(\mu^T) + I)^{1/2} u \rangle_{L^q \times L^p}  \\
& = & \langle (A_q + I)^{1/2} \psi, (A_p(\mu^T) + I)^{1/2} u \rangle_{L^q \times L^p}
= \langle (\ca_q + I) \psi, u \rangle_{W^{-1,q}_\fD \times W^{1,p}_\fD}
\end{eqnarray*}
for all $u \in W^{1,p}_\fD$, where we used (\ref{ethmc-001;1}) in the last step.
So $\varphi = (\ca_q + I) \psi$ and $(\ca_q + I)$ is surjective.
Therefore $q \in \fI_\mu$ by Lemma~\ref{lnonauto302}\ref{lnonauto302-1}.

Next let $q \in \fI_\mu$.
We shall show that $W^{1,q}_\fD = \dom A_q^{1/2}$.
Let $\psi \in W^{1,q}_\fD$.
Then $\psi \in W^{1,2}_\fD$ and $(\ca_q + I) \psi \in W^{-1,q}_\fD$.
By the last part of Proposition~\ref{p-3} there exists a $\tau \in L^q$
such that 
$\langle \tau, (A_p(\mu^T) + I)^{1/2} u\rangle_{L^q \times L^p}
= \langle (\ca_q + I) \psi, u \rangle_{W^{-1,q}_\fD \times W^{1,p}_\fD}$
for all $u \in W^{1,p}_\fD$, where $p = q'$.
Let $u \in \dom A_2(\mu^T)$.
Then 
\begin{eqnarray*}
\langle \tau, (A_p(\mu^T) + I)^{1/2} u\rangle_{L^q \times L^p}
& = & \langle (\ca_q + I) \psi, u \rangle_{W^{-1,q}_\fD \times W^{1,p}_\fD}  \\
& = & \langle (\ca_2 + I) \psi, u \rangle_{W^{-1,2}_\fD \times W^{1,2}_\fD}  \\
& = & (\psi, (A_2(\mu^T) + I) u)_{L^2}  \\
& = & \langle \psi, (A_p(\mu^T) + I) u \rangle_{L^q \times L^p}
 .  
\end{eqnarray*}
Since $\dom A_2$ is a core for $A_p$ one deduces that 
\[
\langle \tau, (A_p(\mu^T) + I)^{1/2} u\rangle_{L^q \times L^p}
= \langle \psi, (A_p(\mu^T) + I) u \rangle_{L^q \times L^p}
\]
for all $u \in \dom A_p(\mu^T)$.
Hence $\langle \tau, v \rangle_{L^q \times L^p}
= \langle \psi, (A_p(\mu^T) + I)^{1/2} v \rangle_{L^q \times L^p}$ for all 
$v \in \dom (A_p(\mu^T))^{1/2}$.
This implies that $\psi \in \dom( ((A_p(\mu^T) + I)^{1/2})^* ) = \dom A_q^{1/2}$.

Conversely, let $\psi \in \dom A_q^{1/2}$.
Then $(A_q + I)^{1/2} \psi \in L^q$.
By the last part of Proposition~\ref{p-3} there exists a $\varphi \in W^{-1,q}_\fD$
such that 
$\langle (A_q + I)^{1/2} \psi, (A_p(\mu^T) + I)^{1/2} u \rangle_{L^q \times L^p}
= \langle \varphi, u \rangle_{W^{-1,q}_\fD \times W^{1,p}_\fD}$
for all $u \in W^{1,p}_\fD$.
Since $q \in \fI_\mu$, the operator $\ca_q + I$ is surjective.
Hence there exists a $\tau \in W^{1,q}_\fD$ such that 
$(\ca_q + I) \tau = \varphi$.
Now let $u \in \dom A_2(\mu^T)$.
Then 
\begin{eqnarray*}
(\psi, (A_2(\mu^T) + I) u)_{L^2}
& = & \langle \psi, (A_p(\mu^T) + I) u \rangle_{L^q \times L^p}  \\
& = & \langle (A_q + I)^{1/2} \psi, (A_p(\mu^T) + I)^{1/2} u \rangle_{L^q \times L^p}  \\
& = & \langle \varphi, u \rangle_{W^{-1,q}_\fD \times W^{1,p}_\fD}  \\
& = & \langle (\ca_q + I) \tau, u \rangle_{W^{-1,q}_\fD \times W^{1,p}_\fD}  \\
& = & \langle (\ca_2 + I) \tau, u \rangle_{W^{-1,2}_\fD \times W^{1,2}_\fD}  \\
& = & (\tau, (A_2(\mu^T) + I) u)_{L^2}
\end{eqnarray*}
Since $(A_2(\mu^T) + I)$ is surjective, it follows that $\psi = \tau \in W^{1,q}_\fD$.

`\ref{thmc-001-1.5}'. 
Again write $p = q'$.
First suppose that $q \geq 2$.
Let $\psi \in \dom A_q$.
Then $\psi \in \dom A_q^{1/2} = W^{1,q}_\fD$ by Statement~\ref{thmc-001-1}.
Moreover, $\psi \in \dom A_2$.
If $u \in \dom (A_2(\mu^T))^{1/2}$ then $u \in \dom (A_p(\mu^T))^{1/2}$ and 
(\ref{ethmc-001;1}) gives
\begin{eqnarray*}
\langle (\ca_q + I) \psi,u \rangle_{W^{-1,q}_\fD \times W^{1,p}_\fD}
& = & ((A_2 + I)^{1/2} \psi, (A_2(\mu^T) + I)^{1/2} u)_{L^2}  \\
& = & ((A_2 + I) \psi, u)_{L^2}
= \langle (A_q + I) \psi,u \rangle_{L^q \times L^p}
 .  
\end{eqnarray*}
Hence $\ca_q \psi = A_q \psi \in L^q$.

Conversely, let $\psi \in W^{1,q}_\fD$ and suppose that $\ca_q \psi \in L^q$.
Write $\tau = (\ca_q + I) \psi$.
Then 
\begin{eqnarray*}
\langle \tau, u \rangle_{L^q \times L^p}
& = & \langle (\ca_q + I) \psi,u \rangle_{W^{-1,q}_\fD \times W^{1,p}_\fD}  \\
& = & \langle (A_q + I)^{1/2} \psi, (A_p(\mu^T) + I)^{1/2} u \rangle_{L^q \times L^p}
\end{eqnarray*}
for all $u \in W^{1,p}_\fD = \dom (A_p(\mu^T) + I)^{1/2}$, where we used (\ref{ethmc-001;1}).
It follows that  $(A_q + I)^{1/2} \psi \in \dom ((A_p(\mu^T)^{1/2} + I)^*) = \dom (A_q + I)^{1/2}$.
Hence $\psi \in \dom ((A_q + I)^{1/2} \, (A_q + I)^{1/2}) = \dom A_q$.

Now suppose that $q \leq 2$.
Let $\psi \in \dom A_q$.
Then $\psi \in \dom A_q^{1/2} = W^{1,q}_\fD$ by Proposition~\ref{p-3}.
If $u \in \dom A_p(\mu^T)$ then by the above $u \in W^{1,p}_\fD$ and $A_p(\mu^T) u = \ca_p(\mu^T) u$.
So
\begin{eqnarray*}
\langle A_q \psi,u \rangle_{W^{-1,q}_\fD \times W^{1,p}_\fD}
& = & \langle A_q \psi,u \rangle_{L^q \times L^p}
= \langle \psi,A_p(\mu^T) u \rangle_{L^q \times L^p}
= \langle \psi,\ca_p(\mu^T) u \rangle_{L^q \times L^p}  \\
& = & \langle \psi,\ca_p(\mu^T) u \rangle_{W^{1,q}_\fD \times W^{-1,p}_\fD}
= \langle \ca_q \psi,u \rangle_{L^q \times L^p}
\end{eqnarray*}
and 
\begin{equation}
\langle A_q \psi,u \rangle_{W^{-1,q}_\fD \times W^{1,p}_\fD}
= \langle \ca_q \psi,u \rangle_{L^q \times L^p}  .
\label{ethmc-001;2}
\end{equation}
Since $\dom A_p(\mu^T)$ is dense in $\dom A_p(\mu^T)^{1/2} = W^{1,p}_\fD$,
one deduces that (\ref{ethmc-001;2}) is valid for all $u \in W^{1,p}_\fD$.
So $\ca_q \psi = A_q \psi \in L^q$.

Conversely, suppose that $\psi \in W^{1,q}_\fD$ and $\ca_q \psi \in L^q$.
Let $u \in \dom A_p(\mu^T)$.
Then again by the above
\begin{eqnarray*}
\langle \psi, A_p(\mu^T) u \rangle_{L^q \times L^p}
& = & \langle \psi, \ca_p(\mu^T) u \rangle_{L^q \times L^p}
= \langle \psi, \ca_p(\mu^T) u \rangle_{W^{1,q}_\fD \times W^{-1,p}_\fD}  \\
& = & \langle \ca_q \psi, u \rangle_{W^{-1,q}_\fD \times W^{1,p}_\fD}
= \langle \ca_q \psi, u \rangle_{L^q \times L^p}
 .  
\end{eqnarray*}
So $\psi \in \dom ( (A_p(\mu^T))^* ) = \dom A_q$
 and the proof of Statement~\ref{thmc-001-1.5} is complete.

`\ref{thmc-001-2}'. 
This follows from Statement~\ref{thmc-001-1} and Proposition~\ref{p-2}.
\end{proof}

We next present a few illustrative examples with explicit subsets of the set $\fI_\mu$.
Note that the requirements on the geometry of $\Omega$ and 
the Dirichlet boundary part $\fD$, as well as on the coefficient function $\mu$
is much higher in the examples than in our general assumptions. 

\begin{exam} \label{xnonaut317.3}
Assume that $\Omega$ is a $C^1$-domain and that $\fD =\partial \Omega$ or
$\fD =\emptyset$  (pure Dirichlet or pure Neumann condition).
If $\mu \in \ce$ is uniformly continuous on
$\Omega$, then $\fI_\mu =(1,\infty)$ by \cite[Section~15]{ADN}
 or \cite[pages~156--157]{Mor}.

The conclusion remains true, if there is a $C^1$-subdomain $\Lambda$ with positive distance 
to the boundary, such that $\mu|_\Lambda$ and $\mu|_{\Omega \setminus \overline \Lambda}$
are uniformly continuous, see \cite[Theorem~1.1 and Remark~3.15]{ElschnerRS}.
\end{exam}

\begin{exam} \label{xnonaut317.5}
Assume that $\Omega$ is a Lipschitz graph-domain (see \cite[Definition~1.2.1.1]{Gris}).
There are equivalent terminologies for this notion:
 strong Lipschitz domain in \cite[Section~1.1.8]{Maz} and
$\Omega$ possesses the uniform cone property in \cite[Section~1.2.2]{Gris}.
Suppose that $\mu \in \ce$ takes symmetric matrices as
values.
Then, under the same continuity properties for $\mu$ as in Example~\ref{xnonaut317.3} (both cases), 
the (open) set $\fI_\mu$ contains the interval $[2,3]$ both in the pure Dirichlet case
(that is $\fD =\partial \Omega$), 
and in the pure Neumann case (that is $\fD =\emptyset$), see \cite{ElschnerRS}.
Moreover, one cannot replace $3$ by a larger number, independent of $\Omega$ and $\mu$.
For the pure Dirichlet Laplacian this result was already proved in 
\cite[Theorem~1.1(c) and Theorem~1.2(a)]{JK},
and for the pure Neumann Laplacian in \cite{Zanger}.
\end{exam}

\begin{exam} \label{xnonaut317.8}
In \cite{DisserKaiserRehberg} there are given a huge variety of domains $\Omega \subset \R^3$,
Dirichlet boundary parts $\fD$ and (possibly discontinuous -- even up to
the boundary) elliptic coefficient
functions $\mu$, such that $\fI_\mu$ contains the interval $[2,3]$.
In particular, it is allowed that $\fD \cap \overline {\partial \Omega
\setminus \fD}$ is not empty, i.e.\ the Dirichlet boundary part 
meets the Neumann part.
\end{exam}

For all $q \in (1,\infty)$ we consider the operator $\ca_q$ as a densely defined 
operator in $W^{-1,q}_\fD$ with domain $W^{1,p}_\fD$.

Let $\mu \in \ce$ and $q \in \fI_\mu \cup [2,\infty)$.
Write $p = q'$.
Then it follows from Proposition~\ref{p-3} and Theorem~\ref{thmc-001}\ref{thmc-001-1}
that $(A_p(\mu^T) + I)^{1/2} \colon W^{1,p}_\fD \to L^p$ is a
topological isomorphism.
Let $(( A_p(\mu^T) + I)^{1/2})' \colon L^q \to W^{-1,q}_\fD$ be the adjoint 
of the operator.
Then $(( A_p(\mu^T) + I)^{1/2})'$ is an isomorphism too.
We use the isomorphism $(( A_p(\mu^T) + I)^{1/2})'$ to transfer the 
$C_0$-semigroup $S^{(q)}$ on $L^q$ to a $C_0$-semigroup $T^{(q)}$
on $W^{-1,q}_\fD$.
Explicitly, for all $t \in (0,\infty)$ define 
$T^{(q)}_t \colon W^{-1,q}_\fD \to W^{-1,q}_\fD$
by 
\begin{equation}
T^{(q)}_t = (( A_p(\mu^T) + I)^{1/2})' \, S^{(q)}_t \, 
            \Big( (( A_p(\mu^T) + I)^{1/2})' \Big)^{-1}
 .  
\label{elnonaut304;30}
\end{equation}
Then $T^{(q)}$ is a $C_0$-semigroup on $W^{-1,q}_\fD$.
Clearly $(( A_p(\mu^T) + I)^{1/2})'$ is an extension of the operator
$(A_q + I)^{1/2}$ and hence $\Big( (( A_p(\mu^T) + I)^{1/2})' \Big)^{-1}$ is 
an extension of $(A_q + I)^{-1/2}$.
Since $(A_q + I)^{-1/2}$ and $S^{(q)}_t$ commute for all $t > 0$, it follows
that $T^{(q)}_t$ is an extension of $S^{(q)}_t$ for all $t > 0$.

We denote the generator of $T^{(q)}$ by $- B_q = - B_q(\mu)$.
Obviously $T^{(q_1)}$ is consistent with $T^{(q_2)}$ for all 
$q_1,q_2 \in \fI_\mu \cup [2,\infty)$.
Hence the graph of $B_{q_1}$ is an extension of the graph of $B_{q_2}$
and 
\begin{equation}
\{ (\psi, B_{q_2} \psi) : \psi \in \dom B_{q_2} \}
= \{ (\psi, B_{q_1} \psi) : \psi \in \dom B_{q_2} \} 
      \cap \Big( W^{-1,q_2}_\fD \times W^{-1,q_2}_\fD \Big)
\label{elnonaut304;1}
\end{equation}
if $q_1 \leq q_2$.

\begin{lemma} \label{lnonaut304}
Adopt Assumption~\ref{assu-general}. 
Let $\mu$ be in $\ce$.
\begin{tabel}
\item \label{lnonaut304-1}
If $q \in \fI_\mu \cup [2,\infty)$, then 
$\dom B_q = \dom A_q^{1/2}$.
\item \label{lnonaut304-2}
If $q \in \fI_\mu$, then $B_q = \ca_q$.
\item \label{lnonaut304-3}
If $q \in [2,\infty)$, then $B_q = \widetilde \ca_q$.
\end{tabel}
\end{lemma}
\begin{proof}
`\ref{lnonaut304-1}'.
By definition of the semigroup $T^{(q)}$ it follows that 
\[
\dom B_q
= \{ (( A_p(\mu^T) + I)^{1/2})' v : v \in \dom A_q \}
 .  \]
Write $p = q'$.
If $v \in \dom A_q^{1/2}$ and $u \in W^{1,p}_\fD$, then 
\begin{eqnarray*}
\langle  (( A_p(\mu^T) + I)^{1/2})' v, u \rangle_{W^{-1,q}_\fD \times W^{1,p}_\fD}
& = & (v, (A_p(\mu^T) + I)^{1/2} u)_{L^q \times L^p}  \\
& = & ( (A_q + I)^{1/2} v, u)_{L^q \times L^p}
= \langle (A_q + I)^{1/2} v, u \rangle_{W^{-1,q}_\fD \times W^{1,p}_\fD}
 .  
\end{eqnarray*}
So $(( A_p(\mu^T) + I)^{1/2})' v =  (A_q + I)^{1/2} v$ for all $v \in \dom A_q^{1/2}$.
Hence
\[
\dom B_q
= \{ (( A_p(\mu^T) + I)^{1/2})' v : v \in \dom A_q \}
= \dom A_q^{1/2}
 .  \]
This proves Statement~\ref{lnonaut304-1}.

`\ref{lnonaut304-2}'.
Suppose $q \in \fI_\mu$.
Then $\dom B_q = \dom A_q^{1/2} = W^{1,q}_\fD = \dom \ca_q$ by Statement~\ref{lnonaut304-1},
Theorem~\ref{thmc-001}\ref{thmc-001-1} and Proposition~\ref{p-3}.
If $v \in \dom A_q^{1/2}$ and $u \in W^{1,p}_\fD$, then
\begin{eqnarray*}
\langle  (( A_p(\mu^T) + I)^{1/2})' (A_q + I)^{1/2} v, u \rangle_{W^{-1,q}_\fD \times W^{1,p}_\fD}
& = & ( (A_q + I)^{1/2} v, ( A_p(\mu^T) + I)^{1/2} u)_{L^q \times L^p}  \\
& = & \langle (\ca_q + I) v, u \rangle_{W^{-1,q}_\fD \times W^{1,p}_\fD} ,
\end{eqnarray*}
where the last equality is (\ref{ethmc-001;1}).
So $ (( A_p(\mu^T) + I)^{1/2})' (A_q + I)^{1/2} v = (\ca_q + I) v$
for all $v \in \dom A_q^{1/2}$.
Hence $(( A_p(\mu^T) + I)^{1/2})' (A_q + I) v = (\ca_q + I) (A_q + I)^{1/2} v$
for all $v \in \dom A_q$.
Using again that $(( A_p(\mu^T) + I)^{1/2})' v =  (A_q + I)^{1/2} v$
one deduces that 
\begin{eqnarray*}
\lefteqn{
(\ca_q + I) \Big( (( A_p(\mu^T) + I)^{1/2})' v \Big)  
} \hspace*{10mm} \\*
& = & (\ca_q + I) (A_q + I)^{1/2} v  \\ 
& = & (( A_p(\mu^T) + I)^{1/2})' (A_q + I) v  \\
& = & (( A_p(\mu^T) + I)^{1/2})' (A_q + I) \Big( (( A_p(\mu^T) + I)^{1/2})' \Big)^{-1}
         \Big( (( A_p(\mu^T) + I)^{1/2})' v \Big)  \\
& = & (B_q + I) \Big( (( A_p(\mu^T) + I)^{1/2})' v \Big) 
\end{eqnarray*}
for all $v \in \dom A_q$.
Hence $B_q = \ca_q$.

`\ref{lnonaut304-3}'.
Let $q \in [2,\infty)$.
It follows from (\ref{elnonaut304;1}) and Statement~\ref{lnonaut304-2} that 
\[
\dom B_q 
= \{ \psi \in W^{-1,q}_\fD : B_2 \psi \in W^{-1,q}_\fD \}
= \{ \psi \in W^{-1,q}_\fD : \ca_2 \psi \in W^{-1,q}_\fD \}
= \widetilde \ca_q
  .  \]
So $B_q = \widetilde \ca_q$.
\end{proof}

In (\ref{elnonaut304;30}) the topological isomorphism
$((A_p(\mu^T) + I)^{1/2})'$ was used to define the $C_0$-semigroup $T^{(q)}$
from the $C_0$-semigroup $S^{(q)}$.
It then transfers properties of the generator of $S^{(q)}$ to 
properties of the generator of $T^{(q)}$.

\begin{thm} \label{tnonaut610}
Adopt Assumption~\ref{assu-general}.
Let $\mu \in \ce$ and $q \in \fI_\mu \cup [2,\infty)$.
Then the operator $B_q +I$ satisfies 
maximal parabolic regularity on the space $W^{-1,q}_\fD$.
\end{thm}
\begin{proof}
The semigroup $S^{(q)}$ is a contraction 
semigroup, hence the operator 
$A_q +I$ has maximal parabolic regularity in the space 
$L^q$ by Lamberton \cite{Lamb}.
Since $((A_p(\mu^T) + I)^{1/2})'$ is a topological isomorphism
from $L^q$ onto $W^{-1,q}_\fD$, where $p = q'$, it follows from 
(\ref{elnonaut304;30}) that the operator $B_q +I$ satisfies 
maximal parabolic regularity on the space $W^{-1,q}_\fD$.
\end{proof}

Note that $B_q = \ca_q$ for all $q \in \fI_\mu$ in the next theorem.
The case $q \in [2,\infty)$ in the next corollary has been proved before in 
\cite[Theorem~11.5]{ABHR}.

\begin{cor} \label{cnonaut611}
Adopt Assumption~\ref{assu-general}.
Then for all $c_\bullet,c^\bullet > 0$ there exists a $\delta \in (0,1)$ such 
that the operator $B_q +I$ satisfies 
maximal parabolic regularity on the
space $W^{-1,q}_\fD$ for all $ q \in (2-\delta, \infty)$
and $\mu \in \ce(c_\bullet,c^\bullet)$.
\end{cor}
\begin{proof}
This follows immediately from Proposition~\ref{p-2} and Theorem~\ref{tnonaut610}.
\end{proof}

\begin{rem} \label{r-explain1}
It is clear that there is an asymmetry in the cases $q \in [2,\infty)$ and 
$q \in (1,2]$.
In the first case,
maximal parabolic regularity holds for the operators $\widetilde \ca_q +I$
on $W^{-1,q}_\fD$ by Theorem~\ref{tnonaut610} and Lemma~\ref{lnonaut304}\ref{lnonaut304-3},
even if the domain of this operator is unknown.
On the contrary, in the case  $q<2$,
we can only prove maximal parabolic regularity for the 
operator $\ca_q + I$ in $W^{-1,q}_\fD$ if $q \in \fI_\mu$.
This is a severe restriction 
on $q$, see Examples~\ref{xnonaut317.3}--\ref{xnonaut317.8}. 
It is an open problem whether in Corollary~\ref{cnonaut611} maximal parabolic 
regularity for $\ca_q + I$ is valid on $W^{-1,q}_\fD$ for all $q \in (1,2)$.
\end{rem}

\section{Time dependent coefficients} \label{Snonaut7}

We next consider coefficient functions which also may depend on time.
Let $\mu \colon J \to \ce$ be a function. 
We frequently write $\mu_t = \mu(t)$ for all $t \in J$.
Note that $\mu_t \in L^\infty(\Omega; \Ri^{d \times d}) \subset L^1(\Omega; \Ri^{d \times d})$ for all $t \in J$.
We say that {\bf $\mu$ is $L^1$-measurable} if the map 
$t \mapsto \mu_t$ is measurable as a map from $J$ into $L^1(\Omega; \Ri^{d \times d})$.

In the main result of this section we require  measurability of the coefficient function 
only in the space $L^1(\Omega; \Ri^{d \times d})$. 
This allows that $\mu_t$ is discontinuous in the space variable for each $t \in J$.
Note that the set of point in $\Omega$ where $\mu_t$ is discontinuous may depend on~$t$.
In general the map $t \mapsto \mu_t$ from $J$ into $L^\infty(\Omega; \Ri^{d \times d})$ 
is discontinuous at \emph{every} time point $t$ and therefore it cannot be measurable.
An example is mentioned in the introduction and it will be considered in more detail
in Section~\ref{Snonaut8}.

\begin{lemma} \label{lnonaut321.5}
Adopt Assumption~\ref{assu-general}.
Let $c^\bullet > 0$ and $\mu \colon J \to \bigcup_{c_\bullet > 0} \ce(c_\bullet,c^\bullet)$
be an $L^1$-measurable map.
Let $q,r \in (1,\infty)$.
Then one has the following.
\begin{tabel}
\item  \label{lnonaut321.5-1}
The map $t \mapsto \ca_q(\mu_t) \psi$
is (strongly) measurable from $J$ into $W^{-1,q}_\fD$ for all $\psi \in W^{1,q}_\fD$.
\item  \label{lnonaut321.5-2}
The map $\frac{\partial}{\partial t} + \ca_q(\mu(\cdot)) + I$ is a bounded
linear map from $\mr^r_0(J;W^{1,q}_\fD,W^{-1,q}_\fD)$ into $L^r(J;W^{-1,q}_\fD)$
with norm at most $1 + c^\bullet$.
\end{tabel}
\end{lemma}
\begin{proof}
Let $\psi, \varphi \in C^\infty_\fD(\Omega)$.
Then the map
\[
\rho \mapsto 
\int_\Omega \rho \, \nabla \psi \cdot  \overline {\nabla \varphi} 
\]
is continuous from $L^1(\Omega,\Ri^{d \times d})$ into $\Ci$.
Since $\mu$ is $L^1$-measurable, also the map
\begin{equation}
t \mapsto 
\langle \ca_q(\mu_t) \psi,\varphi \rangle_{W^{-1,q}_\fD \times W^{1,q}_\fD}
\label{lnonaut321.5;1}
\end{equation}
from $J$ into $\Ci$ is well defined, bounded and measurable.
Since $C^\infty_\fD(\Omega)$ is dense in  $W^{1,q}_\fD$ and $W^{1,q'}_\fD$
and $\mu$ is bounded, 
the map (\ref{lnonaut321.5;1}) is measurable for all 
$\psi \in W^{1,q}_\fD$ and $\varphi \in W^{1,q'}_\fD$.
Therefore
one obtains the weak measurability of the map 
$J \ni t \mapsto \ca_q(\mu_t) \psi \in W^{-1,q}_\fD$
for all $\psi \in W^{1,q}_\fD$, which implies also the strong
measurability, since the space $W^{-1,q}_\fD$ is separable. 
This proves the first statement.
The second one is easy.
\end{proof}

The first main result of this section is as follows.
In order to get good estimates, we use again the 
normed spaces $\mr^r_0(J;W^{1,2}_\fD, W^{-1,2}_\fD)\,\widetilde{\;}$ 
introduced in Definition~\ref{d-tilde}. 

\begin{thm} \label{tnonaut701}
Let $c_\bullet,c^\bullet \in (0,\infty)$ with $c_\bullet \leq 1 \leq c^\bullet$.
Let $s \in (2,\infty)$ and put 
\[
C_{\cj,s} =\max_{r \in \{ s,s' \} }\|\bigl ( \frac{\partial }{\partial t}
 +\cj  \bigr )^{-1}\|_{L^r(J;W^{-1,2}_\fD) \to \mr^r_0(J;W^{1,2}_\fD,W^{-1,2}_\fD)}
,  \]
where $\cj \colon W^{1,2}_\fD \to W^{-1,2}_\fD$ is the duality map.
Define
\[
\kappa_s:= \frac{1}{12} \, \frac{1}{1 +2 ( 1+ \frac{1 + c_\bullet + c^\bullet}{c_\bullet}) \max(1,c^\bullet) C_{\cj,s}} 
\quad \text{and} \quad 
r_0:=\bigl ( \frac{1}{2} -\kappa_s (1-\frac{2}{s})\bigl )^{-1}.
\]
Then for every $L^1$-measurable $\mu \colon J \to \ce(c_\bullet,c^\bullet)$ and 
$r \in (r_0',r_0)$,
the family $ \{ \ca_2(\mu_t) + I \} _{t \in J}$ has maximal $L^r(J;W^{-1,2}_\fD)$-regularity 
and 
\[
\|\Big( \frac{\partial}{\partial t} + \ca_2(\mu(\cdot)) + I \Big)^{-1}
   \|_{L^r(J;W^{-1,2}_\fD) \to \mr^r_0(J;W^{1,2}_\fD, W^{-1,2}_\fD)\,\widetilde{\;}}
\leq  8 \frac{1 + c_\bullet + c^\bullet}{c_\bullet}
,  \]
where the norm on $\mr^r_0(J;W^{1,2}_\fD, W^{-1,2}_\fD)\,\widetilde{\;}$ is defined
using the operator $\cj$.
\end{thm}
\begin{proof}
We wish to apply Theorem \ref{thm-e-Hilbert}.
Let $V = W^{1,2}_\fD$. 
Then $V^* = W^{-1,2}_\fD$. 
Note that $\frac{1}{2} = \frac{1-\theta}{s} + \frac{\theta}{s'}$
if $\theta = \frac{1}{2}$. 
Let $\mu \colon J \to \ce(c_\bullet,c^\bullet)$ be 
an $L^1$-measurable map.
For all $t \in J$ define $\gots_t \colon V \times V \to \Ci$ by
\[
\gots_t[\psi,\varphi] 
= \gots_{\mu(t)}[\psi,\varphi] + (\psi,\varphi)_{L^2}
 .  \]
Then $t \mapsto \gots_t[\psi,\varphi]$ is measurable from $J$ into $\Ci$
for all $\psi,\varphi \in V$.
Moreover, $\RRe \mathfrak s_t[\psi,\psi]  \ge c_\bullet\|\psi\|^2_V$ 
and $|\mathfrak s_t[\psi, \varphi]| \le c^\bullet \, \|\psi\|_V \, \|\varphi\|_V$
for all $\varphi,\psi \in V$ and $t \in J$.
If $t \in J$, then $\ca_2(\mu(t)) + I$ is the operator induced by the 
sesquilinear form $\gots_t$.

All the assumptions of Theorem~\ref{thm-e-Hilbert} are satisfied.
If $\tilde \theta \in (\frac{1}{2}-\kappa_s, \frac{1}{2} +\kappa_s)$, then it follows 
from Theorem~\ref{thm-e-Hilbert}\ref{thm-0001-2} that the isomorphism property
is preserved.
Then the assertion follows by using the identity
$\frac{1}{r}=\frac{1-\tilde \theta}{s} +\frac{\tilde \theta}{s'}$.
\end{proof}

The second main result of this section is that non-autonomous maximal 
$L^r(J;W^{-1,q}_\fD)$-regularity extrapolates in both temporal and spatial integrability 
scales, given by $r$ and~$q$.
Again, quantitative estimates as in \eqref{e-kappa} below, are based on the 
constants $C_{\mathcal{K}}^r$ in \eqref{d-C}, corresponding to a suitable 
autonomous reference operator $\mathcal{K}$. 

\begin{thm} \label{t-main}
Suppose Assumption \ref{assu-general} is satisfied. 
Let $c_\bullet,c^\bullet > 0$.
Then there are open intervals $\ci_1,\ci_2 \subset (1,\infty)$ with 
$2 \in \ci_1$ and $2 \in \ci_2$ such that for all $r \in \ci_1$, $q \in \ci_2$ and 
$L^1$-measurable $\mu \colon J \to \ce(c_\bullet,c^\bullet)$
the family  $\{ \ca_q(\mu_t) +I\}_{t \in J}$ has
maximal parabolic $L^r(J;W^{-1,q}_\fD)$-regularity.
So 
\[
\frac{\partial }{\partial t} +\ca_q(\mu(\cdot)) + I
   \colon \mr^r_0(J;W^{1,q}_\fD,W^{-1,q}_\fD)  \to L^r(J;W^{-1,q}_\fD)
\] 
is a topological isomorphism. 
\end{thm}
\begin{proof}
Define $\delta \colon \Omega \to \Ri^{d \times d}$ by $\delta(x) = I$, the identity matrix,
for all $x \in \Omega$.
Then $\delta \in \ce$ and $A(\delta) = - \Delta$, the minus Laplacian.
Let $\cj \colon W^{1,2}_\fD \to W^{-1,2}_\fD$ be the duality mapping.
Then 
$\langle \cj \psi,\varphi \rangle_{W^{-1,2}_\fD \times W^{1,2}_\fD}
= (\psi,\varphi)_{W^{1,2}_\fD}
= (\psi,\varphi)_{L^2(\Omega)} + \sum_{k=1}^d (\partial_k \psi,\partial_k \varphi)_{L^2(\Omega)}
= \langle (\ca_2(\delta) + I) \psi, \varphi \rangle_{W^{-1,2}_\fD \times W^{1,2}_\fD}$
for all $\psi,\varphi \in W^{1,2}_\fD$.
So $\cj = \ca_2(\delta) + I$.

If follows from Proposition~\ref{p-2}, Lemma~\ref{lnonaut304}\ref{lnonaut304-2} 
and Theorem~\ref{tnonaut610} that there exists a $q_0 \in (2,\infty)$ 
such that $\ca_q(\delta) + I \colon W^{1,q}_\fD \to W^{-1,q}_\fD$ is 
an isomorphism and the operator $\ca_q(\delta) + I$ satisfies 
maximal parabolic regularity on the space $W^{-1,q}_\fD$ for all $q \in (q_0',q_0)$.
By Proposition~\ref{p-2}, there is a $q_1 \in (2,q_0]$ such that 
$[q_1',q_1] \subset \fI_{\tilde \mu}$ for all $\tilde \mu \in \ce(c_\bullet,c^\bullet)$.

For all $q\in[q_1',q_1]$, we choose $\mathcal{K}_q = \ca_q(\delta) + I$ as the 
autonomous reference operator for the spaces 
$\mr^r_0(J;W^{1,q}_\fD,W^{-1,q}_\fD)\,\widetilde{\;}$ in Definition~\ref{d-tilde}.
Note that in case $q = 2$ the reference operator is $\ck_2 = \ca_2(\delta) + I = \cj$, 
which was used in Theorem~\ref{tnonaut701}.
For all $s \in (2,\infty)$ and $\alpha \in (0,1)$ with
$\frac{1}{r} = \frac{1-\alpha}{s} + \frac{\alpha}{s'}$ we obtain by 
Lemma~\ref{lemma-t-intgeropo} that 
\begin{equation} \label{inter-rq}
[ \mr^{s'}_0(J;W^{1,q}_\fD,W^{-1,q}_\fD)\,\widetilde{\;}, 
  \mr^s_0(J;W^{1,q}_\fD,W^{-1,q}_\fD)\,\widetilde{\;}\,]_\alpha
= \mr^r_0(J;W^{1,q}_\fD,W^{-1,q}_\fD)\,\widetilde{\;}
\end{equation}
with equality of norms.
Moreover, for each $q\in[q_1',q_1]$ and $r \in (1,\infty)$ the norms
on $\mr^r_0(J;W^{1,q}_\fD,W^{-1,q}_\fD)$ and $\mr^r_0(J;W^{1,q}_\fD,W^{-1,q}_\fD)\,\widetilde{\;}\;$ 
are equivalent.
Hence it suffices to prove the theorem with $\mr^r_0(J;W^{1,q}_\fD,W^{-1,q}_\fD)$ 
replaced by $\mr^r_0(J;W^{1,q}_\fD,W^{-1,q}_\fD)\,\widetilde{\;}$.

By Theorem~\ref{tnonaut701}, there is an $r_0\in(2,\infty)$ such that 
for all $L^1$-measurable $\mu \colon J \to \ce(c_\bullet,c^\bullet)$ and $r \in [r_0',r_0]$ 
the map
\[
\frac{\partial}{\partial t} + \ca_2(\mu(\cdot)) +I 
\colon  \mr^r_0(J;W^{1,2}_\fD,W^{-1,2}_\fD)\,\widetilde{\;} \, \to  L^r(J;W^{-1,2}_\fD)
\]
is a topological isomorphism with  
\begin{equation}\label{beta-est}
\|\Big( \frac{\partial}{\partial t} + \ca_2(\mu(\cdot)) + I \Big)^{-1}
     \|_{L^r(J;W^{-1,2}_\fD) \to \mr^r_0(J;W^{1,2}_\fD, W^{-1,2}_\fD)\,\widetilde{\;}}
\leq 8 \frac{1 + c_\bullet + c^\bullet}{c_\bullet}.
\end{equation}
These will be the important inverse bounds to apply Theorem~\ref{t-sneib}.

Next, let $q \in \{q_1',q_1\}$.
We need a suitable bound on the operator norms
\[
\gamma_{q,r} := \| \frac{\partial}{\partial t} + \ca_q(\mu(\cdot)) + I 
        \|_{\mr^r_0(J;W^{1,q}_\fD,W^{-1,q}_\fD)\,\widetilde{\;} \, \to L^r(J;W^{-1,q}_\fD)}
,
\]
uniformly in $r \in [r_0', r_0]$.
Since both the spaces $\mr^r_0(J;W^{1,q}_\fD,W^{-1,q}_\fD)\,\widetilde{\;}$ and 
$L^r(J;W^{-1,q}_\fD)$ form exact complex interpolation scales in $r$  
by \eqref{inter-rq} and Proposition~\ref{p-interLpp},
it follows by interpolation that 
\[
\gamma_{q,r} \leq \max_{r \in \{r_0',r_0 \}} \gamma_{q,r}. 
\]
Now let $r \in \{ r_0',r_0 \} $.
Then it follows from Lemma~\ref{lnonaut321.5}\ref{lnonaut321.5-2} and (\ref{ed-tilde;1})
that 
\begin{eqnarray*}
\| (\frac{\partial}{\partial t} + \ca_q(\mu(\cdot)) + I) u\|_{L^r(J;W^{-1,q}_\fD)}
& \leq & (1 + c^\bullet) \, \|u\|_{\mr^r_0(J;W^{1,q}_\fD,W^{-1,q}_\fD)}  \\
& \leq & (1 + c^\bullet) \, C_{\ck_q}^r \, \|u\|_{\mr^r_0(J;W^{1,q}_\fD,W^{-1,q}_\fD)\,\widetilde{\;}}
\end{eqnarray*}
for all $u \in \mr^r_0(J;W^{1,q}_\fD,W^{-1,q}_\fD)$.
So $\gamma_{q,r} \leq (1 + c^\bullet) \, C_{\ck_q}^r$.
Set 
\[
\gamma_0 
:= \max_{q \in \{ q_1',q_1 \} }  \max_{r \in \{ r_0',r_0 \} } (1 + c^\bullet) \, C^r_{\mathcal{K}_q}
 .  \]
Then we proved that 
\begin{equation}
\| \frac{\partial}{\partial t} + \ca_q(\mu(\cdot)) + I 
        \|_{\mr^r_0(J;W^{1,q}_\fD,W^{-1,q}_\fD)\,\widetilde{\;} \, \to L^r(J;W^{-1,q}_\fD)}
\leq \gamma_0
\label{et-main;10}
\end{equation}
for all $r \in [r_0',r_0]$, $q \in \{ q_1',q_1 \} $ and $L^1$-measurable
$\mu \colon J \to \ce(c_\bullet,c^\bullet)$.

Let 
\begin{equation}\label{e-kappa}
\kappa = \frac{1}{12} \, \frac{1}{1 +2 ( 1+8 \frac{1 + c_\bullet + c^\bullet}{c_\bullet})  \gamma_0} 
\quad \text{and} \quad 
q_2 = \bigl( \frac{1}{2} -\kappa (1-\frac{2}{q_1}) \bigl)^{-1}.
\end{equation}
Finally, let  $\mu \colon J \to \ce(c_\bullet,c^\bullet)$
be $L^1$-measurable. 
Let $r \in [r_0',r_0]$ and $q \in [q_1',q_1]$.
Then there exists a $\tilde \theta \in [\frac{1}{2} - \kappa,\frac{1}{2} + \kappa]$
such that $\frac{1}{q} = \frac{1 - \tilde \theta}{q_1} + \frac{\tilde \theta}{q_1'}$.
Note that  $\frac{1}{2} = \frac{1 - \theta}{q_1} + \frac{\theta}{q_1'}$
with $\theta = \frac{1}{2}$.
We apply Theorem \ref{t-sneib} with 
$F_1=\mr^r_0(J;W^{1,q_1'}_\fD,W^{-1,q_1'}_\fD)\,\widetilde{\;}$, 
$F_2=\mr^r_0(J;W^{1,q_1}_\fD,W^{-1,q_1}_\fD)\,\widetilde{\;}$, 
$Z_1 = L^r(J;W^{-1,q_1'}_\fD)$, $Z_2 = L^r(J;W^{-1,q_1}_\fD)$ and $\theta = \frac{1}{2}$.
Note that we have the estimates (\ref{et-main;10}) and (\ref{beta-est}).
Since $|\tilde \theta - \frac{1}{2}| \leq \kappa$ one deduces from 
Theorem~\ref{t-sneib} that 
\[
\frac{\partial }{\partial t} +\ca_q(\mu(\cdot)) + I
   \colon \mr^r_0(J;W^{1,q}_\fD,W^{-1,q}_\fD)\,\widetilde{\;} \, \to L^r(J;W^{-1,q}_\fD)
\] 
is a topological isomorphism. 
This completes the proof of Theorem~\ref{t-main}.
\end{proof}

\section{Quasilinear equations} \label{Snonaut00}

In this section we are interested in quasilinear, non-autonomous equations 
of the form
\[
u'(t) - \nabla \cdot \big( \sigma(u(t)) \mu_t \nabla u(t)\big) + u(t) = f(t); \quad u(0)=0.
\]
The main result is the following. 

\begin{theorem} \label{tquasilin}
Let $\Omega \subset \R^d$ be a bounded open set and $\fD \subset \partial \Omega$
be closed.
Suppose Assumption \ref{assu-general} is satisfied. 
Let $c_\bullet,c^\bullet > 0$ and $\mu \colon J \to \ce(c_\bullet,c^\bullet)$
an $L^1$-measurable map.
Let $\sigma_\bullet,\sigma^\bullet \in (0,\infty)$ with $\sigma_\bullet \leq \sigma^\bullet$.
Let $\sigma \colon \R \mapsto [\sigma_\bullet,\sigma ^\bullet]$ be a
continuous function.
Then there exists an $r_0 \in (2,\infty)$ such that for all $r\in (2,r_0)$
and $f \in L^r(J;W^{-1,2}_\fD)$ there exists a 
$u \in L^r(J;W^{1,2}_\fD) \cap W_0(J;W^{-1,2}_\fD)$ such that 
\begin{equation} \label{etquasilin-1}
u'(t) + \ca_2(\sigma(u(t)) \, \mu_t) u(t) + u(t) = f(t) 
\end{equation} 
in $W^{-1,2}_\fD$ for almost every $t \in J$.
\end{theorem} 
\begin{proof}
By Theorem~\ref{tnonaut701} there exist $r_0 \in (2,\infty)$
and $\beta' > 0$ such that for every $L^1$-measurable 
$\mu \colon J \to \ce(\sigma_\bullet \, c_\bullet, \sigma^\bullet \, c^\bullet)$
the family $ \{ \ca_2(\mu(t)) + I \} _{t \in J}$ has maximal 
$L^r(J;W^{-1,2}_\fD)$-regularity and 
\[
\|\Big( \frac{\partial}{\partial t} + \ca_2(\mu(\cdot)) + I \Big)^{-1}
    \|_{L^r(J;W^{-1,2}_\fD) \to \mr^r_0(J;W^{1,2}_\fD, W^{-1,2}_\fD)}
\leq \beta'
\]
for all $r \in [2,r_0]$.

Now let $r \in (2,r_0]$, $f \in L^r(J;W^{-1,2}_\fD)$ and 
$\mu \colon J \to \ce(c_\bullet,c^\bullet)$ be an $L^1$-measurable map.
We wish to define a map $\Psi \colon C(\overline J;L^2) \to C(\overline J;L^2)$.
Let $v \in C(\overline J;L^2)$.
Then $\sigma(v(t,\cdot)) \, \mu_t(\cdot) \in \ce(\sigma_\bullet \, c_\bullet, \sigma^\bullet \, c^\bullet)$
for almost every $t \in J$ and $t \mapsto \sigma(v(t)) \, \mu_t$ is $L^1$-measurable.
Hence there exists a unique $u \in \mr^r_0(J;W^{1,2}_\fD, W^{-1,2}_\fD)$ such that 
\[
u'(t) + \ca_2(\sigma(v(t)) \, \mu_t) u(t) + u(t)  = f(t) 
\]
for almost every $t \in J$.
Then $u \in C(\overline J;L^2)$ by Proposition~\ref{pnonaut410}\ref{pnonaut410-4}.
Define $\Psi(v) = u$.
Then $\Psi(C(\overline J;L^2)) \subset \mr^r_0(J;W^{1,2}_\fD, W^{-1,2}_\fD)$ 
is relatively compact in $C(\overline J;L^2)$ by Proposition~\ref{pnonaut410}\ref{pnonaut410-4}.
We next show that $\Psi$ is continuous.
Then the theorem follows from Schauder's fixed point theorem.

Let $v,v_1,v_2,\ldots \in C(\overline J;L^2)$ and suppose that 
$\lim_{n \to \infty} v_n = v$ in $C(\overline J;L^2)$.
For all $n \in \Ni$ let $u_n = \Psi(v_n)$ and $u = \Psi(v)$.
Then $u_n,u \in \mr^r_0(J;W^{1,2}_\fD, W^{-1,2}_\fD)$,
\[
u_n'(t) + \ca_2(\sigma(v_n(t)) \, \mu_t) u_n(t) + u_n(t)  = f(t) 
\]
and 
\[
u'(t) + \ca_2(\sigma(v(t) \, \mu_t) u(t) + u(t)  = f(t) 
\]
for almost every $t \in J$ and all $n \in \Ni$.
Subtracting gives 
\begin{eqnarray}
\lefteqn{
(u - u_n)'(t) + \ca_2(\sigma(v(t)) \, \mu_t) ( (u - u_n)(t) )
   + (u - u_n)(t)
} \hspace{50mm} \label{etquasilin-2}  \\*
& = & \ca_2(\sigma(v_n(t)) \, \mu_t) u_n(t)
   - \ca_2(\sigma(v(t)) \, \mu_t) u_n(t)
\nonumber
\end{eqnarray}
for almost every $t \in J$ and all $n \in \Ni$.
Since $\lim v_n = v$ in $C(\overline J;L^2)$, also 
$\lim v_n = v$ in $L^2(J;L^2) = L^2(J \times \Omega;\Ci)$.
Hence passing to a subsequence, if necessary, we may assume that 
$\lim_{n \to \infty} v_n(t,x) = v(t,x)$ for almost every $(t,x) \in J \times \Omega$.
For all $n \in \Ni$ define $\tilde u_n \in L^r(J;W^{-1,2}_\fD)$ by 
\[
\tilde u_n(t) 
= \ca_2(\sigma(v_n(t)) \, \mu_t) u_n(t)
   - \ca_2(\sigma(v(t)) \, \mu_t) u_n(t)
 .  \]
We shall show that $\lim \tilde u_n = 0$ weakly in $L^r(J;W^{-1,2}_\fD)$.
Let $w \in L^{r'}(J;W^{1,2}_\fD)$.
Then 
\begin{eqnarray*}
\lefteqn{
| \langle \tilde u_n, w \rangle_{ L^r(J;W^{-1,2}_\fD) \times L^{r'}(J;W^{1,2}_\fD) } |
} \hspace*{5mm} \\*
& = & \bigg| \int_0^T \int_\Omega
   \Big( \sigma(v_n(t,x)) - \sigma(v(t,x)) \Big) 
      \Big( \mu_t(x) \, \nabla u_n(t,x) \Big) \cdot \overline{ \nabla w(t,x) } \, dx \, dt  \\
& \leq & c^\bullet 
   \int_0^T 
   \Big( \int_\Omega \Big| \big( \sigma(v_n(t,x)) - \sigma(v(t,x)) \big) \, \nabla w(t,x) \Big|^2 \, dx \Big)^{1/2}
   \Big( \int_\Omega |\nabla u_n(t)|^2 \Big)^{1/2} \, dt  \\
& \leq & c^\bullet \Big( \int_0^T \Big( \int_\Omega \Big| 
\big( \sigma(v_n(t,x)) - \sigma(v(t,x)) \big) \, \nabla w(t,x) \Big|^2 \, dx \Big)^{r'/2} \Big)^{1/r'}
   \|u_n\|_{L^r(J;W^{1,2}_\fD)}
 .  
\end{eqnarray*}
Obviously 
$\|u_n\|_{L^r(J;W^{1,2}_\fD)} 
\leq \|u_n\|_{\mr^r_0(J;W^{1,2}_\fD, W^{-1,2}_\fD)}
\leq \beta' \, \|f\|_{L^r(J;W^{-1,2}_\fD)}$
for all $n \in \Ni$.
Hence $\lim_{n \to \infty} | \langle \tilde u_n, w \rangle_{ L^r(J;W^{-1,2}_\fD) \times L^{r'}(J;W^{1,2}_\fD) } | = 0$
by the Lebesgue dominated convergence theorem.
So $\lim \tilde u_n = 0$ weakly in $L^r(J;W^{-1,2}_\fD)$.
But 
\[
\Big( \frac{\partial }{\partial t} +\ca_2(\sigma(v(t,\cdot)) \mu_t) + I \Big) (u - u_n)(t)
= \tilde u_n(t)
\]
for almost every $t \in J$ and all $n \in \Ni$ by (\ref{etquasilin-2}).
Also $u - u_n \in \mr^r_0(J;W^{1,2}_\fD, W^{-1,2}_\fD)$ for all $n \in \Ni$.
Hence by maximal parabolic regularity $\lim_{n \to \infty} u - u_n = 0$
weakly in $\mr^r_0(J;W^{1,2}_\fD, W^{-1,2}_\fD)$.
In addition the embedding of $\mr^r_0(J;W^{1,2}_\fD, W^{-1,2}_\fD)$ into 
$C(\overline J;L^2)$ is compact by Proposition~\ref{pnonaut410}\ref{pnonaut410-4}.
So $\lim_{n \to \infty} u - u_n = 0$ in $C(\overline J;L^2)$
and the continuity of $\Psi$ follows.
\end{proof}

\begin{corollary} \label{c-morereg}
If the right hand side $f$ in \eqref{etquasilin-1}  belongs to a space 
$L^r(J;W^{-1,q}_\fD)$
and $r, q>2$ are sufficiently close to $2$, then every solution $u$ provided by the 
theorem belongs to the space
$\mr^r_0(J;W^{1,q}_\fD, W^{-1,q}_\fD)$.
\end{corollary}
\begin{proof}
The coefficient function $t \mapsto \sigma (u(t)) \mu_t$ satisfies the assumptions of
 Theorem \ref{t-main}.
\end{proof}

It is unclear whether \eqref{etquasilin-1} has a unique solution.

\section{Example with non-smooth coefficients in space and time} \label{Snonaut8}

Let $\Omega$ and $\fD$ satisfy Assumption~\ref{assu-general}.
Let $\Omega_0 \subset \Omega$ be an open non-empty set such that 
$\overline{\Omega_0} \subset \Omega$.
Define a reference coefficient function $\mu_0 \colon J \to \Ri^{d \times d}$
by 
\[
\mu_0(x)
= \left\{ \begin{array}{ll}
   1 \, I_d & \mbox{if }  x \in \overline {\Omega_0},  \\[10pt]
   2 \, I_d & \mbox{if }  x \in \Omega \setminus  {\Omega_0} ,
          \end{array} \right.
\]
where $I_d$ denotes the identity matrix in $\R^d$.
For all $t\in J$
let $\Psi_t \colon \Omega \to \Omega$ be a map 
and set $\mu_t = \mu_0 \circ \Psi_t$.
Suppose that $t \mapsto \mu_t$ is $L^1$-measurable.
Then by Theorem~\ref{t-main} the operator family $\{\ca(\mu_t) + I\}_{t \in J}$
satisfies maximal parabolic $L^s(J;W^{-1,q}_\fD)$-regularity for all $s,q \in (1,\infty)$
sufficiently close to $2$. 

More specifically, consider the case in which $\Omega_t = \Psi_t(\Omega_0)$ is an
open subset of $\Omega$ for all $t\in J$
such that $\Omega_{t_1} \neq \Omega_{t_2}$ for all 
$t_1, t_2 \in J$ with $t_1 \neq t_2$.
Then the map $t\mapsto \mu_t$ from $J$ into $L^\infty(\Omega;\Ri^{d \times d})$ is discontinuous at 
every point $t\in J$ and it is straightforward to show that also the map 
$t\mapsto \ca_2(\mu_t) + I$ from $J$ into $\cl(W^{1,2}_\fD ; W^{-1,2}_\fD)$ 
is discontinuous.
We remark that the analysis for this problem is known to be complicated already in 
case of elliptic equations, see 
the discussion in \cite{ElschnerKaiserRehbergSchmidt} for 
relatively simple geometries of interfaces.
Moreover, it represents a challenge also in numerics, see e.g.\ \cite{AdamsLi}.
In many interesting cases, the movement 
of the subdomain $\Omega_t$ is not determined by an `outer' law, but may depend on the 
underlying physical/chemical process itself.
Mathematically,
this leads to a free boundary problem where for example $\Psi_t$ depends on the solution $u$. 
Particular (simple) cases may then be covered by Theorem~\ref{tquasilin}
to obtain existence and regularity of a solution.

\section{Concluding remarks} \label{Snonaut9}

\begin{rem} \label{r-real}
It is possible to carry our results over to real spaces: 
in case of the real space $W^{1,q}_{\fD, \R}$ one identifies its dual with the elements of
$W^{-1,q'}_\fD$ which take real values for real functions from 
$W^{1,q}_{\fD}$.
Then one applies the `complex' result.
This is enabled by
the fact that, in case of real coefficients, the corresponding operators map the real
subspace onto the `real' subspace of the image.
  \end{rem}

\begin{rem} 
We expect that our abstract results in Section~\ref{Snonaut4} have further applications 
in the field of maximal parabolic regularity for non-autonomous parabolic equations. 
For example, 
 one could investigate non-autonomous problems in the $X =L^p(\Omega)$-setting, 
cf.\ \cite[Section~5]{ADLO},
 \cite{Fackler1} and \cite{Fackler2}.
Maximal parabolic regularity for autonomous elliptic, second-order
divergence-form operators $A$ on $L^p(\Omega)$, with $p \in (1,\infty)$, 
can be shown under Assumption~\ref{assu-general}\ref{assu-general:i}.
In this case, however, it is very difficult to determine the exact domain 
$D(t)$ of an operator $A(t)$, if the coefficient function is spatially discontinuous.
The condition 
$D(t) = D(0)$ in our results then generically excludes settings like the one in Section~\ref{Snonaut8}.
 At the same time, recent optimal results in \cite{Fackler1}
on varying domains with $D(t)$ 
 require some continuity in time and regularity in space which also do not cover this setting. 
In particular, we highlight that in \cite{Fackler2} it is shown that the extrapolation 
for Lions' result, Proposition~\ref{p-Lions}, is impossible in an $L^p$-setting. 
\end{rem}

\begin{rem}
The abstract results in Section~\ref{Snonaut4} and applications to non-autonomous forms in
 Section~\ref{Snonaut3} naturally include systems of equations. 
For the more specific setting in 
Sections~\ref{Snonaut5}--\ref{Snonaut7}, one has the required 
elliptic $W^{1,q}$-regularity for systems  (see \cite[Section~6]{HJKR} 
or \cite[Section~7]{BMMM}),
but presently  the corresponding maximal parabolic regularity results
are an open problem.
\end{rem}

\small
\noindent
{\sc K. Disser,
Weierstrass Institute for Applied Analysis and Stochastics,
Mohrenstr.~39, 
10117 Berlin, 
Germany}  \\
{\em E-mail address}\/: {\bf disser@wias-berlin.de}

\mbox{}

\noindent
{\sc A.F.M. ter Elst,
Department of Mathematics,
University of Auckland,
Private bag 92019,
Auckland 1142,
New Zealand}  \\
{\em E-mail address}\/: {\bf terelst@math.auckland.ac.nz}

\mbox{}

\noindent
{\sc J. Rehberg,
Weierstrass Institute for Applied Analysis and Stochastics,
Mohrenstr.~39, 
10117 Berlin, 
Germany}  \\
{\em E-mail address}\/: {\bf rehberg@wias-berlin.de}

\end{document}